\documentclass[pdflatex]{sn-jnl}

\usepackage{graphicx}%
\usepackage{multirow}%
\usepackage{amsmath,amssymb,amsfonts}%
\usepackage{amsthm}%
\usepackage{mathrsfs}%
\usepackage[title]{appendix}%
\usepackage{xcolor}%
\usepackage{textcomp}%
\usepackage{manyfoot}%
\usepackage{booktabs}%
\usepackage{algorithm}%
\usepackage{algorithmicx}%
\usepackage{algpseudocode}%
\usepackage{listings}%

\theoremstyle{thmstyleone}%
\newtheorem{theorem}{Theorem}
\newtheorem{proposition}[theorem]{Proposition}%

\newtheorem{corollary}{Corollary}

\newtheorem{lemma}{Lemma}

\theoremstyle{thmstyletwo}%
\newtheorem{remark}{Remark}%

\theoremstyle{thmstylethree}%
\newtheorem{definition}{Definition}%

\begin{document}

\title[On a general method for deriving a fourth-order differential equation satisfied by Laguerre–Hahn orthogonal polynomials with new results for the class 0 analogous to Hermite]{On a general method for deriving a fourth-order differential equation satisfied by Laguerre–Hahn orthogonal polynomials with new results for the class 0 analogous to Hermite}


\author[1]{\fnm{Mohamed} \sur{Khalfallah}}\email{mohamed.khalfallah@fsm.rnu.tn}

\author[2]{\fnm{Pascal} \sur{Maroni}} 
\equalcont{17th January 1933 -- 16th January  2024}

\author*[3]{\fnm{Zélia} \sur{da Rocha}}\email{mrdioh@fc.up.pt}

\affil[1]{\orgdiv{Department of Mathematics}, 
\orgname{Faculty of Sciences of Monastir, University of Monastir}, 
\orgaddress{
\city{Monastir}, 
\postcode{5019}, 
\country{Tunisia}}}

\affil[2]{\orgdiv{Laboratoire Jacques-Louis Lions}, 
\orgname{Sorbonne Université, CNRS}, 
\orgaddress{
\street{Boite courrier 187; 4, place Jussieu}, 
\city{Paris}, 
\postcode{75252 Paris cedex 05}, 
\country{France}}}

\affil[3]{\orgdiv{Departamento de Matemática, Centro de Matemática da Universidade do Porto (CMUP)}, 
\orgname{Faculdade de Ciências da Universidade do Porto}, 
\orgaddress{
\street{Rua do Campo Alegre n. 687}, 
\city{Porto}, 
\postcode{4169-007}, 
\country{Portugal}}}


\abstract{In this work, we develop a constructive method for deriving four structure relations and a fourth-order linear differential equation satisfied by Laguerre–Hahn orthogonal polynomial sequences. The method relies on a combination of structure relations, their successive derivatives, and algebraic elimination techniques. Particular attention is given to semiclassical and classical families, which are recovered as special cases within this general framework. The approach is systematized in the form of an algorithm. Using symbolic computations, we obtain explicit new results for Laguerre–Hahn polynomials of class zero, analogous to the Hermite case. In addition, we present results for a semiclassical example of class~1.}




\maketitle


\section{Introduction}



The investigation of orthogonal polynomial sequences $\displaystyle\{P_n\}_{n\geq 0}$ satisfying differential equations of the form
\begin{equation}\label{Eq-Diff-general}
\sum_{k=0}^N f_k\, y^{(k)} = 0,
\end{equation}
where the coefficients $f_k$ are polynomials (possibly depending on $n$), is closely related to measure perturbation theory and to the spectral analysis of differential operators \cite{Buendia-1988}. This topic has been widely studied in the literature on special functions and possesses a long-standing history. For a detailed overview of the developments in this area, we refer the interested reader to the survey \cite{Everitt-1991}. It was established in \cite{Hahn-1978} that whenever an orthogonal polynomial sequence satisfies a differential equation of the form \eqref{Eq-Diff-general}, the order of such an equation can always be reduced to the minimal cases $N=2$ or $N=4$.\\

In the particular case $N=2$ in \eqref{Eq-Diff-general}, where $f_2$ and $f_1$ do not depend on $n$ and $f_0=\lambda$, with $\lambda$ denoting a spectral (eigenvalue) parameter depending on $n$, the differential equation takes the form
\begin{equation}\label{Eq-Diff-classical}
f_2\, y'' + f_1\, y' + \lambda y = 0.
\end{equation}
In this situation, the classification of orthogonal polynomial solutions is well known. More precisely, the corresponding sequence 
$\displaystyle\{P_n\}_{n\geq 0}$ must coincide, up to a linear change of the independent variable, with one of the classical families of orthogonal polynomials, namely the Hermite, Laguerre, Jacobi, or Bessel polynomials (see \cite{Bochner-1929} and also \cite{Everitt-1991} for a survey of the problem concerning the determination of orthogonal polynomial families satisfying \eqref{Eq-Diff-classical}).\\

In the case $N=2$, and more generally, it is well established—at least since the works \cite{Hahn-1983} and \cite{Hahn-1978} (see also \cite{Maroni-1991})—that a sequence of orthogonal polynomials is semiclassical if and only if it satisfies a second-order differential equation with polynomial coefficients whose degrees are bounded.\\

For $N=4$, it is known that Laguerre-Hahn orthogonal polynomials satisfy fourth-order differential equations with polynomial coefficients. However,  in contrast with the second-order case, such differential equations do not yet yield a complete characterization of the Laguerre-Hahn class, since the converse implication is still open in general. It is conjectured that every orthogonal polynomial sequence satisfying a fourth-order differential equation with polynomial coefficients of bounded degrees necessarily belongs to the Laguerre-Hahn class (see \cite{Magnus-1983}, \cite[Sect.III]{brezinski1985polynomes}).
 \\

In the present work, we concentrate on differential equations associated with Laguerre--Hahn orthogonal polynomials. These polynomials are closely connected with Stieltjes functions that satisfy Riccati equations with polynomial coefficients 
\cite{Alaya-these-1996, Alaya-Maroni,Dini-these-1988,Dzoumba-these-1985,Magnus-1983, Maroni-1983, Maroni-1991},
\begin{equation}\label{intro-Riccati}
\Phi S' = B S^2 + C S + D, \qquad \Phi \neq 0.
\end{equation}

The notion of Laguerre--Hahn linear functionals was developed through a detailed formalism of the required algebraic operations together with an appropriate topological framework; see 
\cite{Dini-these-1988, Maroni-1991}. 
More precisely, \cite[Theorem~3.1]{Dini-these-1988} establishes the equivalence between the Riccati equation \eqref{intro-Riccati} satisfied by the formal Stieltjes function
\[
S(z):=-\sum_{n\ge0}\frac{(u)_n}{z^{n+1}},
\]
where $(u)_n$ denotes the $n$-th moment of the associated linear functional $u$, and the functional equation
\begin{equation}\label{intro-Riccati2}
(\Phi u)' + \psi u + B(x^{-1}u^2)=0, 
\qquad \psi = -\Phi' - C.
\end{equation}
Moreover, in \cite[Chapter IV, Theorem 1.1]{Dini-these-1988} is presented a structure relation characterizing Laguerre-Hahn polynomials (see \eqref{Dini_St_Rel}).\\

From a structural and constructive perspective, the Laguerre--Hahn family occupies a prominent position in the theory of orthogonal polynomials. Indeed, within the extensive catalogue of orthogonal polynomial systems, most of the orthogonal polynomial sequences studied in the literature turn out to belong to this class. In particular, either of the equations \eqref{intro-Riccati} or \eqref{intro-Riccati2} may reduce to the semiclassical situation under suitable conditions. Specifically, when $B=0$, one recovers the semiclassical case; the classical families arise when, in addition, $\deg(\Phi)\leq 2$ and $\deg(\psi)=1$ \cite{Maroni-1987}.  On the other hand, the semiclassical case may also occur even when $B$ is not identically zero. This phenomenon appears, for instance, for Stieltjes functions associated with second-degree linear functionals \cite{Maroni-1995}. More generally, a similar situation can arise in the context of third-degree classes. In fact, every linear functional of degree three belongs to the Laguerre--Hahn class \cite{Salah-Maroni-2000}, although the converse statement does not hold in general.\\

Laguerre--Hahn orthogonal polynomials can be obtained through suitable perturbations of the Stieltjes function associated with semiclassical orthogonal polynomials, or alternatively by modifying the coefficients in the second-order recurrence relation of such polynomials \cite{Askey-1984,Dehesa-1990,Maroni-1991,Ronveaux-1990}. Consequently, several well-known families of Laguerre--Hahn polynomials arise in this framework. These include, for instance, the associated polynomials related to semiclassical orthogonal polynomials \cite{Askey-1984,Belmehdi-Ronveaux-1991,Bustoz-1982,Wimp-1987}, as well as the co-recursive, co-dilated, and co-modified polynomial sequences \cite{Belmehdi-1989,Letessier-1994,Ronveaux-1990}. Despite the variety of known examples, the general problem of classifying Laguerre--Hahn orthogonal polynomials remains largely open.\\

For the strict Laguerre--Hahn orthogonal polynomials, i.e., the non-semiclassical ones, the minimal order of a differential equation of the form \eqref{Eq-Diff-general} is equal to four. In \cite{Dzoumba-these-1985}, such a fourth-order differential equation is expressed as a determinant of order five. Later, the authors in \cite{Dini-these-1988,Ronveaux-1990} obtained an alternative representation in which the fourth-order differential equation is written as a determinant of order three. See \cite{Alaya-these-1996,Bouakkaz-these} for explicit representations of the differential equation.
In this context, the works of A. Ronveaux on Laguerre--Hahn orthogonal polynomials are mainly devoted to perturbations of semiclassical orthogonal polynomials, with particular emphasis on the fourth-order differential equations satisfied by these polynomial sequences.  Within this research framework, several types of perturbations of semiclassical orthogonal polynomials have been investigated in \cite{Belmehdi-1989,  Marcellan-1989, Ronveaux-1988, Ronveaux-1991, Ronveaux-1990, Ronveaux-1989, Ronveaux-1995}.\\

The main objective of the present work is to develop a general method for deriving a fourth-order differential equation satisfied by Laguerre-Hahn orthogonal polynomials. The proposed approach is based on a systematic combination of structure relations, their derivatives, and algebraic elimination techniques, leading to a compact representation of the differential equation. This method unifies and generalizes several previous results and provides a systematic framework for deriving differential equations for both strict Laguerre-Hahn and semiclassical families.\\

It should be noted that Laguerre–Hahn linear functionals of class zero have been described in \cite{Bouakkaz-these,Bouakkaz-Maroni-1991}, both via the functional equation and the second-order recurrence relation. In the examples considered here in this work, we focus specifically on Laguerre–Hahn polynomials of class zero analogous to Hermite, for which the fourth-order differential equation has not yet been established in the literature. We correct here the coefficients of the structure relations given in \cite{Bouakkaz-these}, and, as an application of our general theoretical framework, we derive the corresponding fourth-order differential equation.\\

The manuscript is organized as follows. In Section~\ref{Section2}, we recall the necessary definitions and fundamental results concerning orthogonal polynomials, moment functionals, and Laguerre–Hahn forms. Section~\ref{Section3} is devoted to the derivation of a fourth-order linear differential equation satisfied by Laguerre–Hahn polynomials. More precisely, we first establish a set of structure relations (Theorem~\ref{lemma-4RS}) and then prove the main differential equation (Theorem~\ref{proposition-Main}). Corollaries for the semiclassical and classical cases are also discussed. 
In Section~\ref{Section4}, we systematize the constructive method as an algorithm. 
In the next two sections, we present results obtained with a symbolic implementation of that algorithm in {\it Mathematica$^{\circledR}$} for three specific families.
In Section~\ref{Section5}, we provide the four structure relations and the corresponding fourth-order differential equation for two cases of Laguerre–Hahn sequences of class zero, analogous to the Hermite family, given in \cite{Bouakkaz-these,Bouakkaz-Maroni-1991}. The classical Hermite sequence is recovered as a particular case of case 1, for which we give the structure relations and the differential equations corresponding to the semiclassical setting. The following section is devoted to an almost symmetric semiclassical family of class 1 given in \cite{Maroni-Mejri-2011} for which we provide the differential equations.  This article finishes with some conclusions.

\section{Notation and basic background}\label{Section2} 

In this section, we present some basic definitions, notations, and results that are used throughout this paper.
\subsection{Basic tools}
Let $\mathcal{P}$ denote the vector space of polynomials with complex coefficients, and let $\mathcal{P'}$ be its algebraic dual space.  
The elements of $\mathcal{P'}$ will be referred to as \emph{forms} (or linear functionals).  
The pairing between $\mathcal{P}$ and $\mathcal{P'}$ is expressed through the duality brackets $\langle \cdot, \cdot \rangle$.  
For a form $u \in \mathcal{P'}$, the sequence of complex numbers $(u)_n,~ n \geq 0$, is called the \emph{moment sequence} of $u$ relative to the monomial basis $\{x^n\}_{n \geq 0}$.  
In particular, the $n$-th moment is given by $(u)_n := \langle u, x^n \rangle$, 
so that $u$ is uniquely determined by the sequence of its moments.  

In the following, we shall refer to a sequence $\{P_n\}_{n \geq 0}$ as a \emph{polynomial sequence} (PS) if $\deg P_n = n$ for all $n \geq 0$. A \emph{monic polynomial sequence} (MPS) is a PS in which each polynomial has a leading coefficient equal to one.  
If ${\{P_{n}\}}_{n \geq 0}$ is a MPS, there exists a unique sequence
$\{u_n\}_{n\geq 0}$, $u_n\in\mathcal{P}^{\prime}$, called the dual sequence of $\{P_{n}\}_{n\geq 0}$, such that,
\begin{equation}\label{SucDual}
\langle u_{n},P_{m}\rangle=\delta_{n,m}, \quad n,m\geq 0.
\end{equation}
We say that a sequence of forms $\{v_n \}_{ n \geq 0}$ is normalised if and only if $\left( v_n  \right)_n = 1, $ $n \geq 0$, and if $n \geq 1$, then $\left( v_n  \right)_m = 0, $ $m=0, \ldots, n-1$. Thus, the dual sequence  $\{u_n\}_{n\geq 0}$ is normalised. The first form $u_0$ is called the canonical form of ${\{P_{n}\}}_{n \geq 0}$.

We now introduce some operations on $\mathcal{P'}$ following \cite{Maroni-1991}.  
For $c \in \mathbb{C},~ f,p \in \mathcal{P}$, and $u \in \mathcal{P'}$, we define
\begin{align*}
&\langle fu, p \rangle = \langle u, fp \rangle, \quad
\langle u', p \rangle = -\langle u, p' \rangle,  \\
&\langle (x-c)^{-1}u, p \rangle = \langle u, \theta_c p \rangle
= \left\langle u, \frac{p(x)-p(c)}{x-c} \right\rangle.
\end{align*}
Given $f \in \mathcal{P}$ and $u \in \mathcal{P'}$, the product $uf$ is defined by  $(u f)(x) := \left\langle u, \displaystyle\frac{x f(x)-\zeta f(\zeta)}{x-\zeta} \right\rangle $.\\
This definition allows us to introduce the \emph{Cauchy product} of two forms $u,v \in \mathcal{P'}$ by 
\[
\langle uv, f \rangle := \langle u, v f \rangle, \quad f \in \mathcal{P}.
\]

In addition, we make use of the formal Stieltjes function associated with 
$u \in \mathcal{P}'$, defined by  \cite{Maroni-1991}
\[
S(u)(z) := - \sum_{n \geq 0} \frac{(u)_n}{z^{n+1}},
\]
which provides an alternative representation of the moment sequence $\{(u)_n\}_{n \geq 0}$.  
Since the moments uniquely determine $u$, the function $S(u)(z)$ does so as well.

\medskip

A linear functional $u$ is called \emph{regular} (or \emph{quasi-definite}) if there exists a sequence of polynomials  $\{P_n\}_{n \geq 0}$ such that \cite{Chihara-1978}
\begin{equation*}
\langle u, P_n P_m \rangle = r_n \, \delta_{n,m}, \quad n, m \geq 0,
\end{equation*}
where $\{r_n\}_{n \geq 0}$ is a sequence of nonzero complex numbers and $\delta_{n,m}$ denotes the Kronecker symbol.
The sequence $\{P_n\}_{n\geq 0}$ is then said orthogonal with respect to $u$. 
Then, necessarily, $\{P_n\}_{n \geq 0}$ is a PS, $u=(u)_0u_0$, and $\{P_n\}_{n \geq 0}$ and $u$ can be normalized. In the sequel, we shall consider that each $P_n(x)$ is monic, and $(u)_0=1$ (i.e. $u=u_0$).
Henceforth, a monic orthogonal
polynomial sequence $\{P_n\}_{n\geq 0}$ will be indicated as MOPS. 

It is well known that an MOPS is characterized by the following second-order linear recurrence relation and initial conditions, respectively \cite{Chihara-1978}
\begin{eqnarray}
&& P_0(x)=1\  ,\quad  P_1(x)=x-\beta_0, \label{ic_TTRR}\\
&& P_{n+2}(x)=(x-\beta_{n+1})P_{n+1}(x)-\gamma_{n+1}P_{n}(x),~~n\geq 0, \label{TTRR}
\end{eqnarray}
being $\{\beta_n\}_{n\geq 0}$ and $\{\gamma_{n+1}\}_{n\geq 0}$ sequences of complex numbers such that $\gamma_{n+1}\neq0$ for all $ n\geq 0$.

Let  $\{P_n^{(1)}\}_{n\geq0}$ be the associated polynomial sequence of order one of the MPS  $\{ P_n\}_{n\geq0}$ with respect to the canonical form $u=u_0$. It is well known that \cite{Chihara-1978}
$$
 P_n^{(1)}(x):=(u\theta_0P_{n+1})(x)=\bigg\langle  u,\frac{P_{n+1}(x)- P_{n+1}(\xi)}{x-\xi}  \bigg\rangle.
$$
The Stieltjes function of $u^{(1)}$ is expressed in terms of that of $u$ as  \cite{Maroni-1991} 
$$
\gamma_1 S\left(u^{(1)}\right)(z)=-\frac{1}{S(u)(z)}-\left(z-\beta_0\right).
$$
More generally, the sequence of associated polynomials of order $(r+1)$, $r\geq 1$, is  defined by recursion
$$
P_n^{(r+1)}(x)=\big(  P_n^{(r)} \big)^{(1)}(x),\quad  u_n^{(r+1)}=\big(  u_n^{(r)} \big)^{(1)},\quad n, r \geq 0.
$$

If $\{ P_n\}_{n\geq0}$ is a MOPS with respect to the form $u$, then for $r\in\mathbb{N}$, the associated sequence of polynomials of order $r$,  $\left\{ P_n^{(r)}\right\}_{n\geq0}$, is also orthogonal with respect to the form $u^{(r)}$ and satisfies the following recurrence relation 
\begin{eqnarray}
&&P_0^{(r)}(x)=1, \quad P_1^{(r)}(x)=x-\beta_0^{(r)},\label{ic_ASSTTRR}\\
&&P_{n+2}^{(r)}(x)=(x-\beta_{n+1}^{(r)})P_{n+1}^{(r)}(x)-\gamma_{n+1}^{(r)}P_n^{(r)}(x),\quad n\geq0,  \label{ASSTTRR}
\end{eqnarray}
where 
\begin{equation*}
\beta_{n}^{(r)}=\beta_{n+r}, \quad \gamma_{n+1}^{(r)}=\gamma_{n+1+r},\quad n\geq0.  
\end{equation*}

We recall the definition of the $r$-perturbed sequence $\{\widetilde{P}_n\}_{n\geq 0}$, for a fixed integer $r\geq 0$, associated with a MOPS $\{P_n\}_{n\geq 0}$, as introduced in \cite{Maroni-1991}. It is an MOPS satisfying the following second-order recurrence relation 
\[
\begin{array}{l}
\widetilde{P}_0(x)=1,\quad \widetilde{P}_1(x)=x-\widetilde{\beta}_0,\\[2pt]
\widetilde{P}_{n+2}(x)=(x-\widetilde{\beta}_{n+1})\widetilde{P}_{n+1}(x)-\widetilde{\gamma}_{n+1}\widetilde{P}_n(x),\quad n\geq0,
\end{array}
\]
with
\[
\begin{aligned}
&\widetilde{\beta}_0 = \beta_0 + \mu_0,\\[2pt]
&\widetilde{\beta}_n = \beta_n + \mu_n,\quad \mu_n\in\mathbb{C};\qquad 
\widetilde{\gamma}_n = \lambda_n\gamma_n,\quad \widetilde{\gamma}_n\in\mathbb{C}\setminus\{0\},\quad 1\leq n\leq r,\\[2pt]
&\widetilde{\beta}_n = \beta_n,\quad \widetilde{\gamma}_n = \gamma_n,\quad n\geq r+1.
\end{aligned}
\]
We assume that either $\mu_r\neq 0$ or $\lambda_n\neq 1$. The so-called co-recursive case corresponds to a perturbed case of order $0$. Using the notations $\mu:=(\mu_1,\ldots,\mu_r)$, $\lambda:=(\lambda_1,\ldots,\lambda_r)$, $r\geq 1$, we write
\[
\widetilde{P}_n(x)=P_n\left(\mu_0;\,{\mu\atop\lambda}\,;r;x\right),\qquad n\geq0,
\]
and the sequence $\{\widetilde{P}_n\}_{n\geq0}$ is orthogonal with respect to the perturbed form 
$$\widetilde{u}:=u\left(\mu_0;\,{\mu\atop\lambda}\,;r\right).$$

\subsection{Laguerre-Hahn forms}

\begin{definition} \cite{Dzoumba-these-1985,Magnus-1983,Maroni-1983}
A regular form $u$, with $(u)_0=1$, is said to be a Laguerre-Hahn form if its formal Stieltjes function satisfies the Riccati equation
\begin{equation}\label{Riccati}
A(z)S'(u)(z)=B(z)S^2(u)(z)+C(z)S(u)(z)+D(z),
\end{equation}
where $A$, $B$, $C$, and $D$ are polynomials.\\
The sequence \(\{P_n\}_{n\geq 0}\) orthogonal with respect to \( u \) is also called a Laguerre-Hahn sequence. 
\end{definition}

\begin{remark} \cite{Maroni-1991}
If $A=0$ identically, the form $u$ is classified as a second-degree form. 
If $A$ is not identically zero, it may be assumed, without loss of generality, 
that it is monic; and we let $A:=\Phi$. Under this normalization, the condition $B \neq 0$ 
characterizes $u$ as a strict Laguerre-Hahn form, whereas the case $B = 0$ 
corresponds to a semiclassical form.
\end{remark}

There are several characterizations of Laguerre-Hahn forms. Some of them are listed in the following result.
\begin{proposition} \cite{Alaya-Maroni,Bouakkaz-Maroni-1991,Dini-these-1988,Maroni-1991} 
Let $u$ be a regular and normalized form, i.e.,
 $(u)_0=1$, and let $\{P_n\}_{n\geq 0}$ be its corresponding MOPS. The following statements are equivalent
\begin{enumerate}
\item[(i)] $u$ is a Laguerre-Hahn form satisfying \eqref{Riccati} with $A=\Phi$.
\item[(ii)]  \cite{Dini-these-1988} $u$ satisfies the functional equation
\begin{equation}\label{Laguerre-Hahn-EF}
(\Phi u)'+\psi u+B(x^{-1}u^2)=0,
\end{equation}
where $\Phi$, $B$, $C$, and $D$ are the polynomials in \eqref{Riccati} and
\begin{align*}
&C=-\Phi'-\psi,\\
&D=-(u\theta_0\Phi)'-(u\theta_0\psi)-(u^2\theta_0^2B).
\end{align*}
\item[(iii)] \cite{Dini-these-1988} Each polynomial \(P_n, n \geq 0\), verifies the so-called structure relation
\begin{equation}\label{Dini_St_Rel}
\Phi(x)P_{n+1}^{\prime}(x) - B(x)P_n^{(1)}(x) = \sum_{\mu=n-s}^{n+d} \theta_{n,\mu} P_\mu(x), \quad n \geq s + 1,
\end{equation}
where \(\Phi\) and \(B\) are the polynomials defined in (i),  \(\{P_n^{(1)}\}_{n \geq 0}\) is the sequence of associated orthogonal polynomials of order 1 of \(\{P_n\}_{n \geq 0}\),  
$d=\max(t,q)$, $s=\max(p-1,d-2)$, being $t$, $p$, and $q$ the degrees of  $\Phi$, $\psi$, and $B$, respectively. 
\end{enumerate}
\end{proposition}

It is worth noting that the above functional equation \eqref{Laguerre-Hahn-EF} is not uniquely determined.
Indeed, if $u$ is a Laguerre--Hahn form and $\chi$ is an arbitrary polynomial, then $u$ also satisfies
\[
(\chi \Phi u)' + \bigl(\chi \psi - \chi' \Phi\bigr) u
+ (\chi B)\,(x^{-1}u^{2}) = 0.
\]
This observation motivates the following definition.
\begin{definition} \cite{Alaya-Maroni, Bouakkaz-Maroni-1991}
The class of a Laguerre-Hahn form $u$ is the non-negative integer number defined as
$$
s:=\min\max\big\{\deg{\psi}-1, \max\{\deg{\Phi}, \deg{B}\}-2\big\},
$$
where the minimum is taken among all polynomials $\Phi, \psi$ and $B$ such that $u$
satisfies \eqref{Laguerre-Hahn-EF}.
\end{definition}

Taking into account that the class of a Laguerre-Hahn form is crucial to state a hierarchy of such families, we need to give a criterion to characterize it.
\begin{proposition}\label{proposition-simplification} \cite{Alaya-Maroni,Bouakkaz-Maroni-1991}
Let $u$ be a Laguerre-Hahn form and let $\Phi$ and $\psi$ be non-zero polynomials 
such that \eqref{Laguerre-Hahn-EF} holds.
 Let
 \begin{equation}\label{s_class_LH}
 s=\max\big\{\deg{\psi}-1, \max\{\deg{\Phi}, \deg{B}\}-2\big\}.
 \end{equation} 
 Then $s$ is the class of $u$ if and only if
\begin{equation*}
\prod_{c\in\mathcal{Z}_{\Phi}}{\Big(|\Phi'(c)+\psi(c)|+|B(c)|+|\langle u, \theta_c^2\Phi+\theta_c\psi+u\theta_0\theta_c B\rangle|\Big)}\neq0,
\end{equation*}
where $\mathcal{Z}_{\Phi}$ denotes the set of zeros of $\Phi$.
\end{proposition}
\begin{remark}
When it is possible to simplify by the factor \( x - c \), we obtain the new functional equation
\[ 
((\theta_c\Phi)u)' + (\theta_c\psi + \theta_c^2\Phi)u + (\theta_cB)(x^{-1}u^2) = 0. 
\]
Then \( u \) is of class less than or equal to \( s - 1 \).    
\end{remark}

Based on Proposition~\ref{proposition-simplification}, one obtains an alternative criterion to determine the class using the polynomials involved in the Riccati equation~\eqref{Riccati}.
\begin{corollary} \cite{Alaya-Maroni}
Let $u$ be a Laguerre-Hahn form and let $A=\Phi$, $B$, $C$, and $D$ be non-zero polynomials satisfying \eqref{Riccati}. Then $s$ given by \eqref{s_class_LH} is the class of $u$ if and only if the polynomials $\Phi$, $B$, $C$, and $D$ are coprime or, equivalently,
$$
\prod_{c\in\mathcal{Z}_{\Phi}}{\big(|B(c)|+|C(c)|+|D(c)|\big)}\neq0.
$$
\end{corollary}

Let $a \in \mathbb{C}\setminus\{0\}$ and $b \in \mathbb{C}$. 
If a Laguerre-Hahn form $u$ of class $s$ satisfies \eqref{Laguerre-Hahn-EF}, 
then the shifted form $\tilde{u} = (h_{a^{-1}} \circ \tau_{-b})u$ 
is also a Laguerre-Hahn form of class $s$ and satisfies \cite{Bouakkaz-these,Bouakkaz-Maroni-1991,Dini-these-1988}
\begin{equation*}
(\tilde{\Phi}\,\tilde{u})' + \tilde{\psi}\,\tilde{u}
+ \tilde{B}\bigl(x^{-1}\tilde{u}^2\bigr) = 0,
\end{equation*}
where 
$$\tilde{\Phi}(x) = a^{-\deg\Phi}\,\Phi(ax+b),\ 
\tilde{\psi}(x) = a^{1-\deg\Phi}\,\psi(ax+b),\ 
\tilde{B}(x) = a^{-\deg\Phi}B(ax+b).$$

Consequently, a shifting transformation preserves both the Laguerre--Hahn character and
the class of the form. As a result, one may work with canonical
functional equations by appropriately relocating the zeros of $\Phi$ in
\eqref{Laguerre-Hahn-EF}.\\

The above theoretical background provides the necessary tools to investigate the differential equations satisfied by Laguerre-Hahn polynomials. In the next section, we develop a systematic approach to obtain a fourth-order differential equation for such families.

\section{A fourth-order linear differential equation} \label{Section3}
This section is dedicated to the derivation of a fourth-order linear differential equation satisfied by any monic orthogonal polynomial sequence belonging to the Laguerre–Hahn class. The demonstration relies on the structure relation \eqref{R4} together with the auxiliary relations \eqref{R1}–\eqref{R3} established in Lemma~\ref{Lemma-P}. The main result is formulated in Theorem~\ref{proposition-Main}, with special attention given to the semiclassical and classical cases, which arise as particular reductions in Corollaries 
\ref{corollary-eq-diff-semiclassical} and \ref{corollary-eq-diff-classical}.

\begin{proposition}\label{proposition-FSR} \cite{Dini-these-1988,Dzoumba-these-1985,Maroni-1991}
Let \(\{P_n\}_{n \geq 0}\) be a MOPS with respect to \(u\), satisfying \eqref{ic_TTRR}-\eqref{TTRR}. The following statements are equivalent.
\begin{enumerate}
\item[(i)] \(u\) is a Laguerre-Hahn form of class \(s\) satisfying the functional equation \eqref{Laguerre-Hahn-EF}.

\item[(ii)] \(\{P_n\}_{n \geq 0}\) satisfies  the following structure relation 
\begin{align}
\Phi(x) P_{n+1}'(x) - B_0(x) P_n^{(1)} (x) =& \frac{1}{2}\big(C_{n+1}(x) - C_0(x)\big) P_{n+1} (x)\nonumber\\ 
&- \gamma_{n+1} D_{n+1} (x) P_n (x), \quad n \geq 0, \label{R4}    
\end{align}
where \(C_n\) and \(D_n\) are polynomials with coefficients depending on \(n\), such that 
$$\deg C_n \leq s + 1, \qquad \deg D_n \leq s,$$ 
satisfying the recurrence relations
\begin{align}
C_{n+1}(x) =& -C_n(x) + 2(x - \beta_n) D_n(x), \label{SR-1}\\
\gamma_{n+1} D_{n+1} (x) =& -\Phi(x) + \gamma_n D_{n-1} (x) - (x - \beta_n) C_n (x) + (x - \beta_n)^2 D_n (x), \label{SR-2}
\end{align}
for every \(n \geq 0\), with the initial conditions 
$$
B_0(x)=B(x),\quad C_0(x)=C(x),\quad D_0(x)=D(x),\quad D_{-1}(x)=B(x).
$$
\end{enumerate}
\end{proposition} 

In addition to the previous structure relation~\eqref{R4}, the following lemma gathers the remaining key relations required to establish our main result.

\begin{lemma}\label{Lemma-P} \cite{Alaya-these-1996,Dini-these-1988,Dzoumba-these-1985}
Let \(\{P_n\}_{n \geq 0}\) be a Laguerre-Hahn MOPS. For every $n\geq0$, we have
\begin{align}
\Phi(x)(P^{(1)}_{n-1}(x))' =& D_{n}(x)P^{(1)}_{n}(x) - \frac{1}{2}\left(C_{n+1}(x) - C_0(x)\right)P^{(1)}_{n-1}(x) - D_0(x)P_{n}(x), \label{R1}\\
\Phi(x)(P_{n}(x))' =& D_{n}(x)P_{n+1}(x) - \frac{1}{2}\left(C_{n+1}(x) + C_0(x)\right)P_{n}(x) + B_0(x)P^{(1)}_{n-1}(x), \label{R2} \\
\Phi(x)(P^{(1)}_{n}(x))' =& \frac{1}{2}\left(C_{n+1}(x) + C_0(x)\right)P^{(1)}_{n}(x) - \gamma_{n+1}D_{n+1}(x)P^{(1)}_{n-1}(x) - D_0(x)P_{n+1}(x). \label{R3}
\end{align}
\end{lemma}

The following theorem plays an important role in proving our main result. 

\begin{theorem}\label{lemma-4RS}
Let \(\{P_n\}_{n \geq 0}\) be a Laguerre-Hahn MOPS. The following relations hold, for every $n\geq0$, 
\begin{align}
G_{0,1}(x;n)P^{(1)}_{n-1}(x)+G_{1,1}(x;n)P^{(1)}_{n}(x)+H_{1}(x;n) P_{n}(x)=F_{1}(x;n), \label{(4.2.5.1)S1}\\
G_{0,2}(x;n)P^{(1)}_{n-1}(x)
+ G_{1,2}(x;n)P^{(1)}_{n}(x)
+ H_{2}(x;n)P_{n}(x)
= F_{2}(x;n), \label{(4.2.5.1)S2}\\
G_{0,3}(x;n)P^{(1)}_{n-1}(x)+G_{1,3}(x;n)P^{(1)}_{n}(x)+H_{3}(x;n) P_{n}(x)=F_{3}(x;n), \label{(4.2.5.1)S3}\\
G_{0,4}(x;n)P^{(1)}_{n-1}(x)+G_{1,4}(x;n)P^{(1)}_{n}(x)+H_{4}(x;n) P_{n}(x)=F_{4}(x;n), \label{(4.2.5.1)S4}
\end{align}
with
\begin{align}
G_{0,1}(x;n)=&0, \quad n\geq 0,  \label{G01}\\
G_{1,1}(x;n)=&B_{0}(x), \quad n\geq 0,  \label{G11}\\
H_{1}(x;n)=&-\gamma_{n+1}D_{n+1}(x), \quad n\geq 0, \label{H1}\\
F_1(x;n)=&\Phi(x)P'_{n+1}(x)+M_{0,1}(x;n)P_{n+1}(x), \quad n\geq 0,  \label{F1}\\
G_{0,2}(x;n)=&-2\gamma_{n+1}B_{0}(x)D_{n+1}(x), \quad n\geq 0, \label{G02}\\
G_{1,2}(x;n)=&\frac{1}{2}\big{(}C_{n+1}(x)+C_{0}(x)\big{)}B_{0}(x)+\Phi(x)B^{\prime}_{0}(x), \quad n\geq 0, \label{G12}\\
H_{2}(x;n)=&\gamma_{n+1}\left\{\frac{1}{2}\big{(}C_{n+1}(x)+C_{0}(x)\big{)}D_{n+1}(x)-\Phi(x)D^{\prime}_{n+1}(x)\right\}, \quad n\geq 0, \label{H2}\\
F_{2}(x;n)=&\Phi^2(x)P''_{n+1}(x)+M_{1,2}(x;n)P'_{n+1}(x)+M_{0,2}(x;n)P_{n+1}(x), \quad n\geq 0, \label{F2}\\
G_{0,3}(x;n)=&\gamma_{n+1}\Big{\{}D_{n+1}(x)\big{(}C_{n+1}(x)-C_{0}(x)\big{)}B_{0}(x)-3\Phi(x)\big{(}B_{0}(x)D_{n+1}(x)\big{)}^{\prime}\Big{\}},\quad n\geq 0, \label{G03}\\
G_{1,3}(x;n)=&B_{0}(x)\bigg\{\frac{1}{4}\big{(}C_{n+1}(x)+C_{0}(x)\big{)}^{2} -2\gamma_{n+1}D_{n}(x)D_{n+1}(x) \nonumber \\
& +\frac{1}{2}\Phi(x)\big{(}C_{n+1}(x)+C_{0}(x)\big{)}^{\prime}\bigg\} +\Phi(x)B^{\prime}_{0}(x)\big{\{}C_{n+1}(x)+C_{0}(x)+\Phi^{\prime}(x)\big{\}}\nonumber\\
&+\Phi^{2}(x)B^{\prime\prime}_{0}(x), \quad n\geq 0, \label{G13}\\
H_{3}(x;n)=&\gamma_{n+1} \Big{\{}2B_{0}(x)D_{0}(x)+\frac{1}{2}\big{(}C_{n+1}(x)+C_{0}(x)\big{)}^{\prime}\Phi(x)-\frac{1}{4}\big{(}C_{n+1}(x)+C_{0}(x)\big{)}^{2}\Big{\}}\nonumber\\ 
&\times D_{n+1}(x)+\gamma_{n+1} \Phi(x)\Big{(}C_{n+1}(x)+C_{0}(x)-\Phi^{\prime}(x)\Big{)}D^{\prime}_{n+1}(x)\nonumber\\
&-\gamma_{n+1}\Phi^{2}(x)D^{\prime\prime}_{n+1}(x), \quad n\geq 0, \label{H3}\\
F_{3}(x;n)=&\Phi^3(x)P^{(3)}_{n+1}(x)+M_{2,3}(x;n)P''_{n+1}(x)\nonumber\\
&+M_{1,3}(x;n)P'_{n+1}(x)+M_{0,3}(x;n)P_{n+1}(x), \quad n\geq 0, \label{F3}\\
G_{0,4}(x;n)=&-\frac{1}{2}(C_{n+1}(x)-C_0(x))G_{0,3}(x;n)-\gamma _{n+1}D_{n+1}(x)G_{1,3}(x;n) \nonumber\\
&+B_0(x)H_3(x;n)+\Phi(x)G'_{0,3}(x;n),\quad n\geq 0, \label{G04}\\
G_{1,4}(x;n)=&D_n(x)G_{0,3}(x;n)+\frac{1}{2}(C_{n+1}(x)+C_0(x))G_{1,3}(x;n)+\Phi(x)G'_{1,3}(x;n),\quad n\geq 0, \label{G14}\\
H_{4}(x;n)=&-D_0(x)G_{0,3}(x;n)-\frac{1}{2}(C_{n+1}(x)+C_0(x))H_3(x;n)+\Phi(x)H'_3(x;n),\quad n\geq 0, \label{H4}\\
F_{4}(x;n)=&\Phi^4(x)P^{(4)}_{n+1}(x)+M_{3,4}(x;n)P^{(3)}_{n+1}(x)+M_{2,4}(x;n)P''_{n+1}(x) \nonumber\\
& +M_{1,4}(x;n)P'_{n+1}(x)+M_{0,4}(x;n)P_{n+1}(x),\quad n\geq 0, \label{F4} 
\end{align}
where
\begin{align}
M_{0,1}(x;n)=&-\frac{1}{2}\left(C_{n+1}(x)- C_{0}(x)\right), \quad n\geq 0, \label{M01}\\
M_{1,2}(x;n)=&\Phi(x)\left(\Phi'(x)+M_{0,1}(x;n)\right), \quad n\geq 0, \label{M12}\\
M_{0,2}(x;n)=&\Phi(x)M'_{0,1}(x;n)+B_0(x)D_0(x)+\gamma_{n+1}D_n(x)D_{n+1}(x), \quad n\geq 0, \label{M02}\\
M_{2,3}(x;n)=&\Phi(x)\left(M_{1,2}(x;n)+2\Phi(x)\Phi'(x)\right), \quad n\geq 0, \label{M23}\\
M_{1,3}(x;n)=&\Phi(x)\left(M_{0,2}(x;n)+M'_{1,2}(x;n)\right), \quad n\geq 0, \label{M13}\\
M_{0,3}(x;n)=&\Phi(x)M'_{0,2}(x;n)+D_0(x)G_{1,2}(x;n)-D_n(x)H_2(x;n), \quad n\geq 0, \label{M03}\\
M_{3,4}(x;n)=&\Phi(x)\left\{M_{2,3}(x;n)+3\Phi^2(x)\Phi'(x)\right\}, \quad n\geq 0, \label{M34}\\
M_{2,4}(x;n)=&\Phi(x)\left\{M'_{2,3}(x;n)+M_{1,3}(x;n)\right\}, \quad n\geq 0, \label{M24}\\
M_{1,4}(x;n)=&\Phi(x)\left\{M'_{1,3}(x;n)+M_{0,3}(x;n)\right\}, \quad n\geq 0, \label{M14}\\
M_{0,4}(x;n)=&\Phi(x)M'_{0,3}(x;n)+D_0(x)G_{1,3}(x;n)-D_n(x)H_3(x;n), \quad n\geq 0. \label{M04}
\end{align} 
\end{theorem}

\begin{proof}
The proof is carried out by successive differentiation and substitution using the relations stated in Lemma~\ref{Lemma-P}. 

Firstly, equation~\eqref{R4} can be rewritten as
\begin{equation}\label{RS-m}
B_{0}(x)P^{(1)}_{n}(x) - \gamma_{n+1}D_{n+1}(x)P_{n}(x) 
= \Phi(x)P'_{n+1}(x) + M_{0,1}(x;n)P_{n+1}(x), \quad n\geq 0,
\end{equation}
with \(M_{0,1}(x;n)\) given by \eqref{M01}. This is precisely \eqref{(4.2.5.1)S1} with the coefficients \(G_{0,1}(x;n)\), \(G_{1,1}(x;n)\), \(H_{1}(x;n)\), and \(F_{1}(x;n)\) given by \eqref{G01}, \eqref{G11}, \eqref{H1}, and \eqref{F1}, respectively.
 
Next, taking the derivative of~\eqref{RS-m} and multiply both sides of the resulting equation by \(\Phi(x)\), which yields
\begin{align*}
&\Phi(x)B_{0}'(x)P^{(1)}_{n}(x) + \Phi(x)B_{0}(x)\big(P^{(1)}_{n}(x)\big)' \\
&- \gamma_{n+1}\Phi(x)D_{n+1}'(x)P_{n}(x) - \gamma_{n+1}\Phi(x)D_{n+1}(x)P_{n}'(x) \\
&= \Phi(x)\Phi'(x)P'_{n+1}(x) + \Phi^{2}(x)P''_{n+1}(x)\\
&+ \Phi(x)M_{0,1}'(x;n)P_{n+1}(x) + \Phi(x)M_{0,1}(x;n)P'_{n+1}(x), \quad n\geq 0.
\end{align*}
Using the relations \eqref{R2} and \eqref{R3}, the last becomes
\[
\begin{aligned}
&\Phi(x)B_{0}'(x)P^{(1)}_{n}(x)
+ \frac12 B_{0}(x)\big(C_{n+1}(x)+C_{0}(x)\big)P^{(1)}_{n}(x)
- \gamma_{n+1}B_{0}(x)D_{n+1}(x)P^{(1)}_{n-1}(x)\\
&- B_{0}(x)D_{0}(x)P_{n+1}(x) -\gamma_{n+1}\Phi(x)D_{n+1}'(x)P_{n}(x)
-\gamma_{n+1}D_{n+1}(x)D_{n}(x)P_{n+1}(x) \\
&+ \frac12\gamma_{n+1}D_{n+1}(x)\big(C_{n+1}(x)+C_{0}(x)\big)P_{n}(x)  
- \gamma_{n+1}D_{n+1}(x)B_{0}(x)P^{(1)}_{n-1}(x) \\
&= \Phi^{2}(x)P''_{n+1}(x)
+ \Phi(x)\big(\Phi'(x)+M_{0,1}(x;n)\big)P'_{n+1}(x)
+ \Phi(x)M_{0,1}'(x;n) P_{n+1}(x), \quad n\geq 0.
\end{aligned}
\]
By collecting all terms involving \(P^{(1)}_{n-1}(x)\), \(P^{(1)}_{n}(x)\), and \(P_{n}(x)\) on the left-hand side, and transferring all remaining terms to the right-hand side, we obtain the relation~\eqref{(4.2.5.1)S2}, with \(G_{0,2}(x;n)\), \(G_{1,2}(x;n)\), \(H_{2}(x;n)\), and \(F_{2}(x;n)\) given by \eqref{G02}, \eqref{G12}, \eqref{H2}, and \eqref{F2}, respectively.

Upon differentiating~\eqref{(4.2.5.1)S2} and multiplying both sides of the resulting identity by \(\Phi(x)\), we obtain
\[
\begin{aligned}
\Phi(x)G_{0,2}'(x;n)P^{(1)}_{n-1}(x) &+ \Phi(x)G_{0,2}(x;n)\big(P^{(1)}_{n-1}(x)\big)' \\
&+ \Phi(x)G_{1,2}'(x;n)P^{(1)}_{n}(x) + \Phi(x)G_{1,2}(x;n)\big(P^{(1)}_{n}(x)\big)' \\
&+ \Phi(x)H_{2}'(x;n)P_{n}(x) + \Phi(x)H_{2}(x;n)P_{n}'(x) = \Phi(x)F_{2}'(x;n), \quad n\geq 0.
\end{aligned}
\]

By means of formulas~\eqref{R1}--\eqref{R3}, and after grouping the terms containing \(P^{(1)}_{n-1}(x)\), \(P^{(1)}_{n}(x)\), and \(P_{n}(x)\), we arrive at
\begin{align}
&\Big(\Phi(x)G_{0,2}'(x;n) - \tfrac12\big(C_{n+1}(x)-C_{0}(x)\big)G_{0,2}(x;n) \nonumber\\
& - \gamma_{n+1}D_{n+1}(x)G_{1,2}(x;n) + B_{0}(x)H_{2}(x;n)\Big)P^{(1)}_{n-1}(x) \nonumber\\
&+\Big(D_{n}(x)G_{0,2}(x;n) + \Phi(x)G_{1,2}'(x;n) 
   + \tfrac12\big(C_{n+1}(x)+C_{0}(x)\big)G_{1,2}(x;n)\Big)P^{(1)}_{n}(x) \nonumber\\
&+\Big(-D_{0}(x)G_{0,2}(x;n) + \Phi(x)H_{2}'(x;n) 
   - \tfrac12\big(C_{n+1}(x)+C_{0}(x)\big)H_{2}(x;n)\Big)P_{n}(x) \nonumber\\
&= \Phi(x)F_{2}'(x;n) + D_{0}(x)G_{1,2}(x;n)P_{n+1}(x) - D_{n}(x)H_{2}(x;n)P_{n+1}(x), \quad n\geq 0.\label{E2-}
\end{align}
On the other hand, it can be readily checked that
\[
\begin{aligned}
F_{2}'(x;n) =& \Phi^{2}(x)P^{(3)}_{n+1}(x) + \big[2\Phi(x)\Phi'(x) + M_{1,2}(x;n)\big]P''_{n+1}(x) \\
&+ \big[M_{1,2}'(x;n) + M_{0,2}(x;n)\big]P'_{n+1}(x) + M_{0,2}'(x;n)P_{n+1}(x), \quad n\geq 0.
\end{aligned}
\]
Thus, by substituting this expression together with the explicit forms of
\(G_{0,2}(x;n)\), \(G_{1,2}(x;n)\), \(H_{2}(x;n)\), and \(F_{2}(x;n)\) into~\eqref{E2-},
we obtain the desired relation~\eqref{(4.2.5.1)S3}.

Finally, proceeding in the same manner as for the previous relations, we
differentiate~\eqref{(4.2.5.1)S3} and multiply both sides of the resulting
identity by \(\Phi(x)\), we obtain
\begin{align*}
\Phi(x)G_{0,3}'(x;n)P^{(1)}_{n-1}(x) &+ \Phi(x)G_{0,3}(x;n)\big(P^{(1)}_{n-1}(x)\big)' \\
&+ \Phi(x)G_{1,3}'(x;n)P^{(1)}_{n}(x) + \Phi(x)G_{1,3}(x;n)\big(P^{(1)}_{n}(x)\big)' \\
&+ \Phi(x)H_{3}'(x;n)P_{n}(x) + \Phi(x)H_{3}(x;n)P_{n}'(x) = \Phi(x)F_{3}'(x;n), \quad n\geq 0.
\end{align*}
Once again, by means of relations~\eqref{R1}--\eqref{R3}, and upon collecting the terms involving
\(P^{(1)}_{n-1}(x)\), \(P^{(1)}_{n}(x)\), and \(P_{n}(x)\), we deduce that
\begin{align}
&\Big(\Phi(x)G_{0,3}'(x;n) - \tfrac12\big(C_{n+1}(x)-C_{0}(x)\big)G_{0,3}(x;n) \nonumber\\
& - \gamma_{n+1}D_{n+1}(x)G_{1,3}(x;n) + B_{0}(x)H_{3}(x;n)\Big)P^{(1)}_{n-1}(x) \nonumber \\
&+\Big( D_{n}(x)G_{0,3}(x;n) + \Phi(x)G_{1,3}'(x;n) + \tfrac12\big(C_{n+1}(x)+C_{0}(x)\big)G_{1,3}(x;n)\Big)P^{(1)}_{n}(x) \nonumber \\
&+\Big( -D_{0}(x)G_{0,3}(x;n) + \Phi(x)H_{3}'(x;n) - \tfrac12\big(C_{n+1}(x)+C_{0}(x)\big)H_{3}(x;n)\Big)P_{n}(x) \nonumber\\
&= \Phi(x)F_{3}'(x;n) + D_{0}(x)G_{1,3}(x;n)P_{n+1}(x) - D_{n}(x)H_{3}(x;n)P_{n+1}(x), \quad n\geq 0.\label{E3-}
\end{align}
Moreover, it is easy to see that
\[
\begin{aligned}
F_{3}'(x;n) =& \Phi^{3}(x)P^{(4)}_{n+1}(x) + \big[3\Phi^{2}(x)\Phi'(x) + M_{2,3}(x;n)\big]P^{(3)}_{n+1}(x) \\
&+ \big[M_{2,3}'(x;n) + M_{1,3}(x;n)\big]P''_{n+1}(x) \\
&+ \big[M_{1,3}'(x;n) + M_{0,3}(x;n)\big]P'_{n+1}(x) + M_{0,3}'(x;n)P_{n+1}(x), \quad n\geq 0.
\end{aligned}
\]
Thus, substituting this into \eqref{E3-} yields the relation \eqref{(4.2.5.1)S4} with coefficients \(G_{0,4}(x;n)\), \(G_{1,4}(x;n)\), \(H_{4}(x;n)\), and \(F_{4}(x;n)\) given by \eqref{G04}, \eqref{G14}, \eqref{H4}, and \eqref{F4}, respectively. This concludes the proof of the theorem.
\end{proof}

We now present the main result of this section, which establishes the fourth-order linear differential equation satisfied by each polynomial of a Laguerre--Hahn MOPS, expressed in terms of the coefficients of the structure relations.

\begin{theorem}\label{proposition-Main}
A Laguerre–Hahn MOPS $\{P_n\}_{n \geq 0}$ satisfies the following fourth-order linear differential equation 
\begin{equation}\label{4ODE_LH}
\mathcal{A}(x;n)P^{(4)}_{n+1}(x)
+\mathcal{B}(x;n)P^{(3)}_{n+1}(x)
+\mathcal{C}(x;n)P''_{n+1}(x)
+\mathcal{D}(x;n)P'_{n+1}(x)
+\mathcal{E}(x;n)P_{n+1}(x)=0,    
\end{equation}
with
\begin{align}
\mathcal{A}(x;n) =& \Phi^{4}(x)\Delta_4(x;n),\quad n\geq0, \label{A}
\\[6pt]
\mathcal{B}(x;n) =& M_{3,4}(x;n)\Delta_4(x;n) - \Phi^{3}(x) \Delta_3(x;n),\quad n\geq0, \label{B}
\\[6pt]
\mathcal{C}(x;n) =& M_{2,4}(x;n) \Delta_4(x;n) - M_{2,3}(x;n)\Delta_3(x;n) + \Phi^{2}(x)\Delta_2(x;n),\quad n\geq0, \label{C}
\\[6pt]
\mathcal{D}(x;n) =& M_{1,4}(x;n)\Delta_4(x;n) - M_{1,3}(x;n)\Delta_3(x;n) \nonumber\\
&+ M_{1,2}(x;n)\Delta_2(x;n) - \Phi(x)\Delta_1(x;n),\quad n\geq0,
\label{D}
\\[6pt]
\mathcal{E}(x;n) =& M_{0,4}(x;n)\Delta_4(x;n) - M_{0,3}(x;n)\Delta_3(x;n) \nonumber\\
& + M_{0,2}(x;n)\Delta_2(x;n) - M_{0,1}(x;n)\Delta_1(x;n),\quad n\geq0, \label{E}
\end{align}
where
\begin{align}
\Delta_4(x;n) =&
G_{0,1}(x;n)\big(G_{1,2}(x;n)H_3(x;n)-G_{1,3}(x;n)H_2(x;n)\big) \nonumber\\
& - G_{1,1}(x;n)\big(G_{0,2}(x;n)H_3(x;n) -G_{0,3}(x;n)H_2(x;n)\big) \nonumber\\
& + H_1(x;n)\big(G_{0,2}(x;n)G_{1,3}(x;n)-G_{0,3}(x;n)G_{1,2}(x;n)\big), \quad n\geq0, \label{Delta1}
\\[6pt]
\Delta_3(x;n) =&
G_{0,1}(x;n)\big(G_{1,2}(x;n)H_4(x;n)-G_{1,4}(x;n)H_2(x;n)\big) \nonumber\\
& - G_{1,1}(x;n)\big(G_{0,2}(x;n)H_4(x;n) -G_{0,4}(x;n)H_2(x;n)\big) \nonumber\\
& + H_1(x;n)\big(G_{0,2}(x;n)G_{1,4}(x;n)-G_{0,4}(x;n)G_{1,2}(x;n)\big), \quad n\geq0, \label{Delta2}
\\[6pt]
\Delta_2(x;n) =&
G_{0,1}(x;n)\big(G_{1,3}(x;n)H_4(x;n)-G_{1,4}(x;n)H_3(x;n)\big) \nonumber\\
&- G_{1,1}(x;n)\big(G_{0,3}(x;n)H_4(x;n)-G_{0,4}(x;n)H_3(x;n)\big) \nonumber\\
& + H_1(x;n)\big(G_{0,3}(x;n)G_{1,4}(x;n)-G_{0,4}(x;n)G_{1,3}(x;n)\big), \quad n\geq0, \label{Delta3}
\\[6pt]
\Delta_1(x;n) =&
G_{0,2}(x;n)\big(G_{1,3}(x;n)H_4(x;n)-G_{1,4}(x;n)H_3(x;n)\big) \nonumber\\
& - G_{1,2}(x;n)\big(G_{0,3}(x;n)H_4(x;n) -G_{0,4}(x;n)H_3(x;n)\big) \nonumber\\
& + H_2(x;n)\big(G_{0,3}(x;n)G_{1,4}(x;n)-G_{0,4}(x;n)G_{1,3}(x;n)\big), \quad n\geq0, \label{Delta4}
\end{align}
where the polynomials \(G_{0,k}(x;n)\), \(G_{1,k}(x;n)\), \(H_{k}(x;n)\), and \(M_{j,k}(x;n)\),
\(1 \leq k \leq 4\) and \(0 \leq j \leq 3\), are given by
\eqref{G01}--\eqref{M04}.
\end{theorem}

\begin{proof}
The fourth-order differential equation is expressed in determinantal form from the system
\eqref{(4.2.5.1)S1}--\eqref{(4.2.5.1)S4}, that is,
\[ 
\begin{vmatrix}
G_{0,1}(x;n) & G_{1,1}(x;n) & H_1(x;n) & F_1(x;n)\\
G_{0,2}(x;n) & G_{1,2}(x;n) & H_2(x;n) & F_2(x;n)\\
G_{0,3}(x;n) & G_{1,3}(x;n) & H_3(x;n) & F_3(x;n)\\
G_{0,4}(x;n) & G_{1,4}(x;n) & H_4(x;n) & F_4(x;n)
\end{vmatrix}=0,\quad n\geq 0.
\]
Expanding the determinant along the fourth column yields
\[
\Delta_1(x;n)F_{1}(x;n)
- \Delta_2(x;n)F_{2}(x;n)
+ \Delta_3(x;n)F_{3}(x;n)
- \Delta_4(x;n)F_{4}(x;n)=0,\quad n\geq 0,
\]
where the polynomials \(\Delta_k(x;n)\), for \(1 \leq k \leq 4\), are given by
\eqref{Delta1}--\eqref{Delta4}. Thus, the fourth-order differential equation~\eqref{4ODE_LH} is obtained by substituting the expressions of
\(F_k(x;n)\), \(0 \leq k \leq 4\), given by~\eqref{F1}, \eqref{F2}, \eqref{F3}, and~\eqref{F4}.
\end{proof}

The above result provides a general framework for deriving a fourth-order differential equation for any Laguerre-Hahn polynomial sequence. In the following, we examine two important particular cases: the semiclassical and the classical families. In fact, in the semiclassical case, i.e., when $B_0(x) = 0$, all the coefficients $G_{i,j}$  in Theorem~\ref{lemma-4RS} vanish for $i = 0,1$, and $j = 1,2,3,4$. This leads to a second-order differential equation (the minimal order) -- see Corollary~\ref{corollary-eq-diff-semiclassical} -- as well as to higher-order equations of third and fourth order (see Corollary~\ref{3de4de_sc}).


\begin{corollary}\label{corollary-eq-diff-semiclassical}
A semiclassical MOPS $\{P_n\}_{n \geq 0}$ satisfies the following linear second-order differential equation
\[
\mathcal{C}_{2}(x;n)P_{n+1}''(x)
+\mathcal{D}_{2}(x;n)P_{n+1}'(x)
+\mathcal{E}_{2}(x;n)P_{n+1}(x)=0,\quad n\geq0,
\]
with
\begin{enumerate}
\item[(I)]
\begin{eqnarray}
\mathcal{C}_{2}(x;n) & = & \Phi^{2}(x)H_1(x;n),\quad n\geq0, \label{2DF_1_SC}\\
\mathcal{D}_{2}(x;n) & = & M_{1,2}(x;n)H_1(x;n)-\Phi(x) H_2(x;n),\quad n\geq0, \label{2DF_2_SC}\\
\mathcal{E}_{2}(x;n)  & = &  M_{0,2}(x;n)H_1(x;n)-M_{0,1}(x;n)H_2(x;n),\quad n\geq0.\label{2DF_3_SC}
\end{eqnarray}
\item [(II)]
\begin{align*}
\mathcal{C}_{2}(x;n) =&\Phi^{2}(x)\,D_{n+1}(x),\quad n\geq0,\\
\mathcal{D}_{2}(x;n) =&\Phi(x)\Bigl[\bigl(\Phi'(x)+C_{0}(x)\bigr)D_{n+1}(x) -\Phi(x)D_{n+1}'(x)\Bigr],\quad n\geq0,\\
\mathcal{E}_{2}(x;n)=&D_{n+1}(x)\Bigl[\gamma_{n+1}D_{n}(x)D_{n+1}(x) -\frac{1}{4}\bigl(C_{n+1}^{2}(x)-C_{0}^{2}(x)\bigr) \nonumber\\
&-\frac{1}{2}\Phi(x)\bigl(C_{n+1}'(x)-C_{0}'(x)\bigr)\Bigr]\\
&+\frac{1}{2}\bigl(C_{n+1}(x)-C_{0}(x)\bigr)\Phi(x)D_{n+1}'(x),\quad n\geq0.
\end{align*}
\end{enumerate}
\end{corollary}
\begin{proof}
According to Proposition~\ref{proposition-FSR}, the semiclassical MOPS \(\{P_n\}_{n \geq 0}\) satisfies the main structure relation~\eqref{R4} with \(B_0(x)= 0\), from which it follows that 
$$G_{0,1}(x;n) = G_{1,1}(x;n) = G_{0,2}(x;n) = G_{1,2}(x;n) = 0.$$  
Consequently, \eqref{(4.2.5.1)S1} and \eqref{(4.2.5.1)S2} become
\begin{align*}
H_1(x;n)P_n(x)=F_1(x;n),\quad n\geq0,\\
H_2(x;n)P_n(x)=F_2(x;n),\quad n\geq0,
\end{align*}
which necessarily gives
\[
H_1(x;n)F_2(x;n)-H_2(x;n)F_1(x;n)=0,\quad n\geq0.
\]
Substituting the expressions of \(F_1(x;n)\) and \(F_2(x;n)\) given in~\eqref{F1} and~\eqref{F2}, respectively, yields
\begin{eqnarray}
&&\Phi^{2}(x)H_1(x;n) P_{n+1}''(x)
+\big[M_{1,2}(x;n)H_1(x;n)-\Phi(x) H_2(x;n)\big]\,P_{n+1}'(x) \notag\\
&&+\big[M_{0,2}(x;n)H_1(x;n)-M_{0,1}(x;n)H_2(x;n)\big]\,P_{n+1}(x)=0,\quad n\geq0, \notag
\end{eqnarray}
which establishes the second-order differential equation of type (I).
Furthermore, by substituting the explicit expressions of $H_{1}(x;n)$, $H_{2}(x;n)$, $M_{0,1}(x;n)$, $M_{0,2}(x;n)$, and $M_{1,2}(x;n)$ as given in~\eqref{H1}, \eqref{H2}, \eqref{M01}, \eqref{M02}, and~\eqref{M12}, respectively, and simplifying by the nonzero factor $\gamma_{n+1}$, we obtain the second-order differential equation of type (II). This completes the proof.
\end{proof}

\begin{remark}
The two formulations (I) and (II) are equivalent. Form (I) is specifically intended for the needs of the algorithmic implementation (see Section~\ref{Section4}), as it directly follows from the procedure described in Theorems \ref{lemma-4RS} and \ref{proposition-Main}. Form (II) provides an explicit and compact expression, which is more convenient for direct analytical use.    
\end{remark}

The following lemma provides a useful identity that simplifies the expressions arising in the semiclassical differential equation.
\begin{lemma}
The following identity holds,
\begin{equation} \label{R-simp}
\gamma_{n+1} D_n D_{n+1} - \frac14 (C_{n+1}^2 - C_0^2)
= -\Phi \sum_{\nu=0}^n D_\nu, \quad n\geq0. 
\end{equation}
\end{lemma}
\begin{proof}
We proceed by induction on $n$. For $n=0$, \eqref{SR-2} gives $\gamma_1 D_1 = -\Phi + (x-\beta_0)^2 D_0 - (x-\beta_0)C_0$. Multiplying by $D_0$ and using the expression $$
\frac14(C_1^2-C_0^2) = (x-\beta_0)D_0[(x-\beta_0)D_0 - C_0]
$$
from \eqref{SR-1}, the quadratic terms cancel, yielding
\[
\gamma_1 D_0 D_1 - \frac14(C_1^2-C_0^2) = -\Phi D_0,
\]
which is \eqref{R-simp} for $n=0$.\\
Assume \eqref{R-simp} holds for $n-1$:
\begin{equation}\label{H-R}
\gamma_n D_{n-1}D_n - \frac14(C_n^2-C_0^2) = -\Phi\sum_{\nu=0}^{n-1} D_\nu. 
\end{equation}
Multiply \eqref{SR-2} by $D_n$:
\[
\gamma_{n+1} D_n D_{n+1} = -\Phi D_n + \gamma_n D_{n-1}D_n + (x-\beta_n)^2 D_n^2 - (x-\beta_n)C_n D_n.
\]
From \eqref{SR-1}, we have $2(x-\beta_n)D_n = C_{n+1}+C_n$, so $$
(x-\beta_n)^2 D_n^2 = \frac14(C_{n+1}+C_n)^2 \quad \text{and}\quad (x-\beta_n)C_n D_n = \frac12 C_n(C_{n+1}+C_n).
$$
Substituting and simplifying gives
\[
\gamma_{n+1} D_n D_{n+1} = -\Phi D_n + \gamma_n D_{n-1}D_n + \frac14(C_{n+1}^2 - C_n^2).
\]
Subtract $\frac14(C_{n+1}^2-C_0^2)$ from both sides
\[
\gamma_{n+1} D_n D_{n+1} - \frac14(C_{n+1}^2-C_0^2) = -\Phi D_n + \gamma_n D_{n-1}D_n - \frac14(C_n^2-C_0^2).
\]
By \eqref{H-R}, the right-hand side equals $-\Phi D_n - \Phi\displaystyle\sum_{\nu=0}^{n-1} D_\nu = -\Phi\displaystyle\sum_{\nu=0}^{n} D_\nu$, which is \eqref{R-simp} for $n$.
\end{proof}

\begin{remark}
Thanks to the above identity, the second-order differential equation satisfied by semiclassical orthogonal polynomials takes a simpler and well-known form. Indeed,
substituting \eqref{R-simp} into the expression of $\mathcal{E}_2(x;n)$ given in Corollary~\ref{corollary-eq-diff-semiclassical} and dividing the differential equation by $\Phi(x)$, we recover the following second-order differential equation, which is
well known in the literature for characterizing semiclassical orthogonal polynomials \cite{Maroni-1991}
\begin{equation}\label{Eq-diff-Maroni-1991}
{J}(x;n){P}_{n+1}''(x)+{K}(x;n){P}_{n+1}'(x)+{L}(x;n){P}_{n+1}(x)=0,
\quad n\geq0,
\end{equation}
with
\begin{align*}
{J}(x;n)&={\Phi}(x){D}_{n+1}(x),\quad n\geq0,\\
{K}(x;n)&={C}_0(x){D}_{n+1}(x)-\mathcal{W}\big({\Phi},{D}_{n+1}\big)(x),\quad n\geq0,\\
{L}(x;n)&=\mathcal{W}\Big(\frac{1}{2}\big({C}_{n+1}-{C}_{0}\big),{D}_{n+1}\Big)(x)
-{D}_{n+1}(x)\sum_{\nu=0}^{n}{D}_{\nu}(x),\quad n\geq0,
\end{align*}
where $\mathcal{W}(f,g)=f g'-f' g$ denotes the Wronskian.
\end{remark}

As a direct consequence of the previous reduction, we recover the Bochner characterization of classical orthogonal polynomial families satisfying a second-order linear differential equation with polynomial coefficients  \cite{Bochner-1929}.

\begin{corollary}\label{corollary-eq-diff-classical}
A classical MOPS $\{P_n\}_{n \geq 0}$ satisfies the following linear second-order differential equation  
\[
\mathcal{C}(x;n)P_{n+1}''(x)
+\mathcal{D}(x;n)P_{n+1}'(x)
+\mathcal{E}(x;n)P_{n+1}(x)=0,\quad n\geq0,
\]
with
\[
\begin{aligned}
\mathcal{C}(x;n) &= \Phi(x),\\
\mathcal{D}(x;n) &= \Phi'(x)+C_{0}(x)=-\psi(x),\\
\mathcal{E}(x;n) &= -\left[\sum_{\nu=0}^n D_\nu
+\frac12\bigl(C_{n+1}'(x)-C_0'(x)\bigr)\right],\quad n\ge0.
\end{aligned}
\]
\end{corollary}
\begin{proof}
In this case $D_{n+1}(x)=d_{n+1}$ is a nonzero constant with respect to $x$. Therefore, dividing equation \eqref{Eq-diff-Maroni-1991} by $d_{n+1}$ yields the desired relation.
\end{proof}

It is well known that the minimal order of a differential equation satisfied by semiclassical orthogonal polynomials is two \cite{Hahn-1978}, as already illustrated by the second-order equations obtained above. Nevertheless, as a consequence of our general formalism, we also obtain higher-order differential equations of third and fourth order for the same families, as stated in the following result.

\begin{corollary}\label{3de4de_sc}
A semiclassical MOPS $\{P_n\}_{n \geq 0}$ satisfies the following
third-order and fourth-order differential equations    
\[
\mathcal{B}_3(x;n)P_{n+1}'''(x) + \mathcal{C}_3(x;n)P_{n+1}''(x) + \mathcal{D}_3(x;n)P_{n+1}'(x) + \mathcal{E}_3(x;n)P_{n+1}(x) = 0, \quad n\geq0, 
\]
with
\begin{align*}
\mathcal{B}_3(x;n) &= \Phi^3(x)H_1(x;n),\\[4pt]
\mathcal{C}_3(x;n) &= M_{2,3}(x;n)H_1(x;n),\\[4pt]
\mathcal{D}_3(x;n) &= M_{1,3}(x;n)H_1(x;n) - \Phi(x)H_3(x;n),\\[4pt]
\mathcal{E}_3(x;n) &= M_{0,3}(x;n)H_1(x;n) - M_{0,1}(x;n)H_3(x;n);
\end{align*}
\[
\mathcal{A}_4(x;n)P_{n+1}^{(4)}(x) + \mathcal{B}_4(x;n)P_{n+1}'''(x) + \mathcal{C}_4(x;n)P_{n+1}''(x) + \mathcal{D}_4(x;n)P_{n+1}'(x) + \mathcal{E}_4(x;n)P_{n+1}(x) = 0, 
\]
with
\begin{align*}
\mathcal{A}_4(x;n) &= \Phi^4(x)H_1(x;n),\\[4pt]
\mathcal{B}_4(x;n) &= M_{3,4}(x;n)H_1(x;n),\\[4pt]
\mathcal{C}_4(x;n) &= M_{2,4}(x;n)H_1(x;n),\\[4pt]
\mathcal{D}_4(x;n) &= M_{1,4}(x;n)H_1(x;n) - \Phi(x)H_4(x;n),\\[4pt]
\mathcal{E}_4(x;n) &= M_{0,4}(x;n)H_1(x;n) - M_{0,1}(x;n)H_4(x;n),
\end{align*}
where $M_{i,j}$, $0\le i\le 3$, $1\le j\le 4$, and $H_k$, $1\le k\le 4$, are given in Lemma~\ref{lemma-4RS} under the condition $B_0(x)=0$.
\end{corollary}
\begin{proof}
According to Proposition~\ref{proposition-FSR}, the semiclassical MOPS
\(\{P_n\}_{n \geq 0}\) satisfies the first structure relation~\eqref{R4}
with \(B_0(x)=0\). Consequently,
\(G_{i,j}(x;n)=0\), for \(i=0,1\) and \(1\leq j\leq 4\). By combining equations \eqref{(4.2.5.1)S1} and \eqref{(4.2.5.1)S3}, as well as \eqref{(4.2.5.1)S1} and \eqref{(4.2.5.1)S4}, we obtain respectively \(H_1F_3 - H_3F_1 = 0\) and \(H_1F_4 - H_4F_1 = 0\), which,  after substituting the expressions of \(F_1\), \(F_3\), and \(F_4\), yield the desired relations. 
\end{proof}


\section{Algorithm for the symbolic computation of structure relations and differential equations}\label{Section4}

The theoretical developments presented thus far provide a general and constructive method for deriving structure relations and differential equations for Laguerre–Hahn, semiclassical, and classical orthogonal polynomials. 
This method is systematized in the following algorithm, named {\it 4oDELH.nb}, which was implemented in {\it Mathematica$^{\circledR}$}.

We remark that the coefficients of each structure relation do not have common factors, since each structure relation is normalized by the existence of a coefficient equal to a power of the monic polynomial $\Phi(x)$. In contrast, the coefficients of the differential equations may admit common factors. In such cases, the implementation factors them out and returns the corresponding reduced coefficients.


\vspace{0.5cm}

\noindent {\bf Algorithm {\it 4oDELH} } ({\it fourth-order differential equation for Laguerre-Hahn})

\vspace{0.25cm}

\begin{enumerate} 

\item {\bf Input Data}

\vspace{0.25cm}

- Coefficients of the recurrence relation \eqref{ic_TTRR}-\eqref{TTRR}: $\beta_{n}$, $\gamma_{n+1}$, $n\geq 0$.

- Coefficients of the Stieltjes equation \eqref{Riccati}: 
$\Phi(z)$, $B(z)$, $C(z)$, and $D(z)$.

- Coefficients of the structure relation \eqref{R4}: $B_0(x)=B(x)$,  $C_0(x)=C(x)$, $D_0(x)=D(x)$, $C_{n+1}(x)$, $D_{n+1}(x)$, $n\geq 0$.

\vspace{0.25cm}

\item {\bf Computation of structure relations for Laguerre-Hahn, semiclassical, and classical polynomials}

\vspace{0.25cm}

- Coefficients of the four structure relations \eqref{(4.2.5.1)S1}-\eqref{(4.2.5.1)S4} of Theorem \ref{lemma-4RS} are computed from the input data using the relations \eqref{G01}-\eqref{M04}.

\vspace{0.25cm}

\item {\bf Computation of the fourth-order differential equation for strict Laguerre-Hahn polynomials}

\vspace{0.25cm}

- Coefficients of the fourth-order differential equation \eqref{4ODE_LH} of Theorem \ref{proposition-Main} are computed from the coefficients of the four structure relations obtained in step 2 using the relations \eqref{Delta1}-\eqref{Delta4}, and \eqref{A}-\eqref{E}.

\vspace{0.25cm}

\item {\bf Computation of differential equations for semiclassical, and classical polynomials}

\vspace{0.25cm}

- Coefficients of the second-order differential equation of Corollary \ref{corollary-eq-diff-semiclassical} are computed from the coefficients of the four structure relations obtained in step 2 using the relations \eqref{2DF_1_SC}-\eqref{2DF_3_SC}.

\vspace{0.25cm}

- Coefficients of the third-order and fourth-order differential equations of Corollary \ref{3de4de_sc} are computed from the coefficients of the four structure relations obtained in step 2 using the relations given in that corollary.

\vspace{0.25cm}

\item {\bf Computation of reduced coefficients of differential equations}

\vspace{0.25cm}

- Computation of the greatest common divisor of the coefficients of each differential equation, followed by the determination of reduced coefficients by dividing by this divisor.

\vspace{0.25cm}

\item {\bf Computation of orthogonal polynomials}

\vspace{0.25cm}

-  Polynomials $P_n(x)$, and $P^{(1)}_n(x)$, for fixed values of the integers $n$,  are computed using recurrence relations \eqref{ic_TTRR}-\eqref{TTRR}, and \eqref{ic_ASSTTRR}-\eqref{ASSTTRR}, respectively.

\vspace{0.25cm}

\item {\bf Presentation of results}

\vspace{0.25cm}

- Coefficients of the structure relations and differential equations are expressed in the canonical basis, with factorization carried out as far as possible with respect to $n$, and the parameters of each sequence.

\end{enumerate}

\section{Results for Laguerre-Hahn of class 0 families analogous to Hermite} \label{Section5}

In order to illustrate the developed algorithmic approach, we now apply it to two families of Laguerre–Hahn polynomials analogous to Hermite. These families are of special interest, as they belong to class zero of Laguerre-Hahn, and the classical Hermite sequence is recovered as a particular case. We then present the coefficients of the corresponding structure relations, as well as the resulting fourth-order differential equations. Moreover, in the classical case, we present the differential equations of orders three and four. All these results provide explicit expressions that have not previously appeared in the literature. Here, we treat all sequence parameters as symbols. Results for several particular values of parameters are available in the software.
The characteristic elements of these sequences, which serve as input data, are provided in \cite{Bouakkaz-these,Bouakkaz-Maroni-1991}.

\vspace{0.25cm}

\subsection{Case 1 analogous to Hermite}\label{Section5_H_Case1}

\vspace{0.25cm}

In this case, some coefficients in the structure relations and the fourth-order differential equation have initial conditions, expressed using the following notation
$$
\epsilon_n=1-(1-\rho)\delta_{n,0},\quad n\geq 0.
$$
 However, the reduced coefficients of the differential equation are free of initial conditions.
 
 \vspace{0.25cm}
 
\noindent \textbf{Regularity conditions} 
\[
\tau,\ \lambda,\ \rho \in\mathbb{C}, \quad \rho\neq 0, \quad \tau\not=-n, \ n\geq 1.
\]
\noindent {\bf Recurrence coefficients} 
\begin{eqnarray}
&&\beta_{0}=\lambda,\quad \beta_{n+1}=0,\ n\geq 0; \quad 
\gamma_{1}=\rho\frac{\tau+1}{2},\quad \gamma_{n+1}=\frac{n+\tau+1}{2},\ n\geq 1. \notag
\end{eqnarray}
{\bf Coefficients of the Stieltjes equation} 
\begin{eqnarray}
&&\Phi(x)=1,\quad
B(x)=2\frac{\rho-1}{\rho}x^{2}+2\lambda\frac{2-\rho}{\rho}x+1-\rho(\tau+1)-\frac{2\lambda^{2}}{\rho},\notag\\
&&C(x)=2\frac{\rho-2}{\rho}x+\frac{4\lambda}{\rho},\quad
D(x)=-\frac{2}{\rho}. \notag
\end{eqnarray}
{\bf Coefficients of the Laguerre-Hahn structure relation} 
$$
C_{n+1}(x)=-2x,\quad  D_{n+1}(x)=-2,\quad n\geq 0.
$$
{\bf Relation to the classical Hermite form ${\cal H}$} \cite[Proposition 4.1]{Mohamed-Imed-2025} 
$$u_{0}={\cal H}^{(\tau)}\left(\lambda;\,{0\atop\rho}\,;1\right).$$


\noindent Next, we present a list of the results obtained using the software. 

\vspace{0.25cm}

\noindent {\bf First structure relation}
\begin{eqnarray}
G_{0,1}(x;n)P^{(1)}_{n-1}(x)+G_{1,1}(x;n)P^{(1)}_{n}(x)+H_{1}(x;n)P_{n}(x)=&&\notag\\
\Phi(x) P'_{n+1}(x)+M_{0,1}(x;n)P_{n+1}(x),\ n\geq 0.\notag&&
\end{eqnarray}
\begin{eqnarray}
&&G_{0,1}(x;n)=0,\quad 
G_{1,1}(x;n)=B(x),\quad  H_{1}(x;n)=\epsilon_n(n+\tau +1), \notag\\
&& \Phi(x)  =  1,\quad M_{0,1}(x;n)= \frac{2}{\rho}\left( x(\rho -1)+\lambda \right).\notag
\end{eqnarray}
\noindent {\bf Second structure relation}
\begin{eqnarray}
G_{0,2}(x;n)P^{(1)}_{n-1}(x)+ G_{1,2}(x;n)P^{(1)}_{n}(x)+ H_{2}(x;n)P_{n}(x)= &&\notag\\ 
\Phi^2(x)P''_{n+1}(x)+M_{1,2}(x;n)P'_{n+1}(x)+M_{0,2}(x;n)P_{n+1}(x),\ n\geq 0.&&\notag
\end{eqnarray}
\begin{eqnarray}
G_{0,2}(x;n) & = &\frac{2\epsilon_n}{\rho}(n+\tau +1)\Big(2 x^2(\rho -1) -
  2 x\lambda  (\rho -2)- 2 \lambda ^2+\rho-\rho ^2 (\tau+1) \Big), \notag\\
G_{1,2}(x;n) & = &\frac{2 }{\rho^2 }\Big(  -2x^3 (\rho -1)+2 x^2\lambda  (2 \rho -3) +
x\left(6 \lambda ^2-\rho(3 +2 \lambda ^2 ) +\rho ^2 (\tau +3)\right)  \notag\\
&&-\lambda  \left(2 \lambda ^2-3 \rho+\rho ^2 (\tau +2 ) \right)\Big),\notag\\
H_2(x;n) & = & \frac{2\epsilon_n}{\rho} (n+\tau +1)(x-\lambda),\quad 
\Phi^2(x) =\ 1, \notag\\
M_{1,2}(x;n) & = & \frac{2 }{\rho }\left(x(\rho -1)+\lambda \right),\notag\\
M_{0,2}(x;n) & = & \frac{2 }{\rho^2 }\Big( -2 x^2(\rho -1)+2 x\lambda  (\rho -2) +2 \lambda ^2-2 \rho+\rho ^2(n+2 \tau +3)\Big) .\notag 
\end{eqnarray}
\noindent {\bf Third structure relation}
\begin{eqnarray}
G_{0,3}(x;n)P^{(1)}_{n-1}(x)+G_{1,3}(x;n)P^{(1)}_{n}(x)+H_{3}(x;n)P_{n}(x)=&&\notag\\
\Phi^3(x)P^{(3)}_{n+1}(x)+M_{2,3}(x;n)P''_{n+1}(x)+M_{1,3}(x;n)P'_{n+1}(x)+M_{0,3}(x;n)P_{n+1}(x),\ n\geq 0.&&\notag
\end{eqnarray}
\begin{eqnarray}
 G_{0,3}(x;n)& = & \frac{2\epsilon_n}{\rho^2}(\tau +n+1) 
 \Big(4 x^3(\rho -1)^2 -4 x^2\lambda  (\rho -3) (\rho -1)\notag\\
 &&-2  x \left(\rho ^3 (\tau+1)-\rho ^2 (\tau +5)+
 4 \rho(\lambda^2+1) -6\lambda ^2 \right)\notag\\
&& -  \lambda  \left(4 \lambda ^2+\rho ^2( 2\tau +5 )-8 \rho \right)\Big),\notag\\
&& \notag\\
G_{1,3}(x;n)& = & \frac{2}{\rho^3}\Big(4 x^4(\rho -1)-4 x^3\lambda  (3 \rho -4)\notag\\
&&+ 
2 x^2 \left(-12 \lambda ^2+6 \left(\lambda ^2+1\right) \rho +2n\rho^2  \left( 1 - \rho \right)
  +\rho ^2 (\tau -4) -2 \rho ^3 (\tau +1) \right)\notag\\
  &&-2 x\lambda  \left( -8 \lambda ^2+2 \left(\lambda ^2+6\right) \rho+2n \rho^2 \left(2 - \rho \right)
 +\rho ^2 (2 \tau -5) -2\rho ^3 (\tau +1)\right) \notag\\
 &&-4 \lambda ^4+12 \lambda ^2 \rho+\rho ^2 \left(2 \lambda ^2 (\tau -1)-3\right) +n\rho ^2
   \left(4 \lambda ^2 +2 \rho ^2 (\tau +1)-2 \rho \right)\notag\\
 &&+ \rho ^3 (1-\tau )+2 \rho ^4 (\tau +1)^2\Big),\notag\\  
   &&\notag\\ 
H_{3}(x;n)& = & \frac{2\epsilon_n}{\rho^2}(n+\tau +1)\Big(2 x^2(2 \rho -1)  -4 x\lambda 
   (\rho -1)  -\left(2 \lambda ^2-3 \rho+2 \rho ^2 (\tau +1 )\right)\Big),\notag\\  
   \Phi^3(x)& = & 1,\notag\\
M_{2,3}(x;n)& = & \frac{2 }{\rho }\left(x(\rho -1)+\lambda \right),\notag\\
M_{1,3}(x;n)& = &\frac{2 }{\rho^2}\Big(-2 x^2(\rho -1)+2x \lambda  (\rho -2) + 
2 \lambda ^2-3 \rho+ \rho ^2(n+2 \tau +4 )\Big),\notag\\   
M_{0,3}(x;n)& = & \frac{4 }{\rho^3}\Big(2x^3 (\rho -1)-2 x^2\lambda  (2 \rho -3)
  +x\left(-6 \lambda ^2+ \rho(2 \lambda ^2 +5)+ \rho ^2(n-4 )
   \right)\notag\\
 &&+\lambda  \left(2 \lambda ^2-5 \rho+ \rho ^2(-n+2) \right).\notag
\end{eqnarray}
\noindent {\bf Fourth structure relation}
\begin{eqnarray}
G_{0,4}(x;n) P^{(1)}_{n-1}(x)+G_{1,4}(x;n) P^{(1)}_{n}(x)+H_{4}(x;n)P_{n}(x)=&&\notag\\
\Phi^4(x)P^{(4)}_{n+1}(x)+M_{3,4}(x;n) P^{(3)}_{n+1}(x)+M_{2,4}(x;n) P''_{n+1}(x)&&\notag\\
+M_{1,4}(x;n) P'_{n+1}(x)+M_{0,4}(x;n) P_{n+1}(x),\ n\geq 0.&&\notag
\end{eqnarray}
\begin{eqnarray}
G_{0,4}(x;n) & = & \frac{4\epsilon_n}{\rho^3}(n+\tau+1)\Big(
4x^4 (\rho -1) \left(\rho ^2-\rho +1\right)-4 x^3\lambda  (\rho -2) \left(\rho ^2-2 \rho +2\right) \notag\\
 &&-2x^2 \Big(12 \lambda ^2-4\rho( 3
   \lambda ^2 +2) +\rho ^2(4 \lambda ^2  - \tau  +13  ) \notag\\
 &&+n \rho ^2(\rho-1)+\rho ^3(
   \tau -6 )+\rho ^4 (\tau +1) \Big)\notag\\
 &&-x\lambda 
  \left(-16 \lambda ^2+8 \rho ( \lambda ^2+4)+ 2n\rho ^2( 2-\rho) + \rho ^2 (4\tau -33 ) + \rho ^3(5-2 \tau)\right)\notag\\
&&  -4 \lambda ^4-16 \lambda ^2 \rho +\rho ^2 \left(\lambda ^2 (7-2 \tau )+4\right)+n \left(-2 \lambda ^2 \rho ^2+\rho ^3-\rho ^4 (\tau+1)\right)\notag\\
 &&+\rho ^3 (2 \tau -3)-\rho ^4 (\tau +1) (2 \tau +1)\Big),\notag\\
&& \notag\\ 
G_{1,4}(x;n) & = &\frac{4}{\rho^4}\Big(-4 x^5(\rho -1) +4x^4 \lambda  (4 \rho -5) -2 x^3 \big(-20 \lambda ^2 +2n \rho^2\left(\rho ^2-3 \rho +2 \right)\notag\\
  &&+2 \rho ^4 (\tau +1)-6 \rho ^3 (\tau +1)
  +3 \rho ^2 (\tau -2) +2 \left(6 \lambda ^2+5\right) \rho\big)\notag\\
   &&+2  x^2\lambda  \big(-20 \lambda ^2 +
   2n \rho ^2 \left(\rho ^2-6 \rho +6\right)+2 \rho ^4 (\tau +1)-12 \rho ^3 (\tau +1)\notag\\
  && +3 \rho ^2 (3 \tau -4) +2\rho \left(4 \lambda ^2+15\right) \big)
  +x\Big(20 \lambda ^4+\rho ^3 \left(12 \lambda ^2 \tau +12 \lambda ^2+11 \tau -1\right)\notag\\
   &&-3 \rho ^2 \left(6 \lambda ^2 \tau -4 \lambda ^2-5\right)-4 \lambda ^2 \left(\lambda
   ^2+15\right) \rho \notag\\
   && + n \left(-24 \lambda ^2
   \rho ^2+2 \left(6 \lambda ^2+7\right) \rho ^3-4 \rho ^4 (\tau +4)+2 \rho ^5 (\tau +1)\right)+2 \rho ^5 (\tau +1)^2\notag\\
  && -4\rho ^4 (\tau +1) (\tau +4)     \Big)\notag\\
   &&-\lambda  \left(4 \lambda ^4-3 \rho ^2 \left(2 \lambda ^2 \tau -5\right)-20 \lambda ^2 \rho +n
   \left(-8 \lambda ^2 \rho ^2+14 \rho ^3+\rho ^4 (-4 \tau -9)\right)\right.\notag\\
  &&\left.+\rho ^3 (11 \tau +2)+\rho ^4 (17 \tau +9) \right)\Big),\notag\\
H_{4}(x;n) &= &\frac{4}{\rho ^3} (n+\tau +1)\Big(2x^3\left(2 \rho ^2-2 \rho +1\right)-
2x^2\lambda  \left(2 \rho ^2-4 \rho +3\right)\notag\\
&&-x
 \left( -6 \lambda ^2+\rho(4 \lambda ^2+7) -12\rho ^2+2 \rho ^3 (\tau +1) \right)-\lambda  \left(2 \lambda ^2+5 \rho ^2-7 \rho \right)\Big),\notag\\
\Phi^4(x)& = &1,\quad \notag\\
M_{3,4}(x;n)& = & \frac{2 }{\rho }(x(\rho -1)+ \lambda),\notag\\
M_{2,4}(x;n)& = & \frac{2 }{\rho ^2}\Big( -2 x^2(\rho -1)+2 x\lambda  (\rho -2) +
\left(2 \lambda ^2+n \rho ^2+ \rho ^2 (2\tau +5 )-4 \rho \right)\Big),\notag\\
M_{1,4}(x;n)& = & \frac{4}{\rho ^3}\Big(
  2 x^3(\rho -1)-2x^2 \lambda  (2 \rho -3)+
  x\left(-6 \lambda ^2+\rho(2 \lambda ^2 +7) + \rho ^2(n-6 )\right)\notag\\
   &&+
 \lambda  \left(2 \lambda ^2-7 \rho+\rho ^2(-n+3) \right)\Big), \notag\\
M_{0,4}(x;n)& = &\frac{8}{\rho ^4}\Big( 
-2 x^4(\rho -1)+2 x^3\lambda  (3 \rho -4)\notag\\
   &&+x^2
 \left(12 \lambda +n \left(4 \rho ^3-3 \rho^2\right)- \rho ^2 \left(2(\tau -3)+3 \left(2 \lambda ^2+3\right) \right)+4 \rho ^3 (\tau +1)\right)\notag\\
   &&+x\lambda  \left(-8 \lambda ^2+2 \rho\left(\lambda ^2+9\right)  +2n\rho ^2 \left(3 -2 \rho \right)+\rho ^2 (4 \tau -7)-4
   \rho ^3 (\tau +1)\right) \notag\\
&&  -\left( -2 \lambda ^4+9 \lambda ^2 \rho+\rho ^2 \left(\lambda ^2 (2\tau -1)-4\right)+n\rho ^2 \left(3 \lambda ^2 -3 \rho +2 \rho ^2x (\tau +1)\right)\right.\notag\\
&&\left.+\rho ^3 (1-2 \tau )+2 \rho ^4(\tau +1)^2\right)\Big).\notag
\end{eqnarray}
\noindent {\bf The fourth-order differential equation}
\begin{eqnarray}
\mathcal{A}(x;n)P^{(4)}_{n+1}(x)
+\mathcal{B}(x;n)P^{(3)}_{n+1}(x)
+\mathcal{C}(x;n)P''_{n+1}(x)
+\mathcal{D}(x;n)P'_{n+1}(x)&&\notag\\
+\mathcal{E}(x;n)P_{n+1}(x)=0,\ n\geq 0. &&\notag
\end{eqnarray}
\noindent  {\bf Greatest common factor between} $\mathcal{A}$, $\mathcal{B}$, $\mathcal{C}$, $\mathcal{D}$, and $\mathcal{E}$: 
\begin{eqnarray}
 c(x;0)=4 (\tau +1)^2,\quad  c(x;n)= \frac{4 (n+\tau +1)^2}{\rho ^2},\quad n\geq 1. \notag
\end{eqnarray}
$$\widehat{\mathcal{A}}=\mathcal{A}/c,\quad \widehat{\mathcal{B}}=\mathcal{B}/c,\quad \widehat{\mathcal{C}}=\mathcal{C}/c,\quad 
\widehat{\mathcal{D}}=\mathcal{D}/c,\quad \widehat{\mathcal{E}}=\mathcal{E}/c.$$
\begin{eqnarray}
\widehat{\mathcal{A}}(x;n) & = &
   -8x^4 (n+1) (\rho -1)^2+4x^3\lambda  (\rho -1) \left(n (4 \rho -8)+3 \rho -8\right)\notag\\
   &&-4x^2 \Big(\lambda ^2 \left(\rho ^2-10 \rho +12\right)+n \big(\lambda ^2 \left(2 \rho ^2-12 \rho
   +12\right)-2 \rho ^3+4 \rho ^2 \notag\\
   &&+\left(2 \rho ^2-2 \rho ^3\right) \tau -2 \rho
   \big)-\rho ^3+2 \rho ^2+\left(-\rho ^3-\rho ^2+2 \rho \right) \tau -\rho\Big)\notag\\
   &&-2x\lambda\left(\lambda ^2 (6 \rho -16)+n \left(\lambda ^2 (8 \rho -16)+4 \rho ^3-12 \rho ^2+\left(4
   \rho ^3-8 \rho ^2\right) \tau +8 \rho \right)\right.\notag\\
   &&\left.+\left(\rho ^3-8 \rho \right) \tau
   +\rho ^3-5 \rho ^2+4 \rho\right)\notag\\
   &&-8 \lambda ^4+\lambda ^2 \left(4 \rho -7 \rho ^2\right)+\tau  \left(\lambda ^2 \left(-4
   \rho ^2-8 \rho \right)-10 \rho ^3+10 \rho ^2\right)\notag\\
   &&+n \Big(-8 \lambda ^4+\lambda ^2
   \left(8 \rho -8 \rho ^2\right)+\tau  \left(-8 \lambda ^2 \rho ^2-4 \rho ^4+4 \rho
   ^3\right)\notag\\
   && -2 \rho ^4 \tau ^2-2 \rho ^4+4 \rho ^3-2 \rho ^2\Big) -4 \rho ^3 \tau ^2-6
   \rho ^3+12 \rho ^2-6 \rho,\notag
   \end{eqnarray} 
  \begin{eqnarray} 
\widehat{\mathcal{B}}(x;n) & =  & 2\Big(
  16x^3(n+1) (\rho -1)^2 -6x^2 \lambda  (\rho -1)\left(  n (4 \rho -8)+3 \rho -8  \right)\notag\\
   &&+4x \Big(      \lambda ^2 \left(\rho ^2-10 \rho +12\right)+n \big(\lambda ^2 \left(2 \rho ^2-12 \rho
   +12\right)-2 \rho ^3+4 \rho ^2\notag\\
   &&+\left(2 \rho ^2-2 \rho ^3\right) \tau -2 \rho
   \big) -\rho ^3+2 \rho ^2+\left(-\rho ^3-\rho ^2+2 \rho \right) \tau -\rho)\Big)\notag\\
   &&+\lambda\left(  \lambda ^2 (6 \rho -16)+n \left(\lambda ^2 (8 \rho -16)+4 \rho ^3-12 \rho ^2+\left(4
   \rho ^3-8 \rho ^2\right) \tau +8 \rho \right)\right.\notag\\
   &&\left.+\left(\rho ^3-8 \rho \right) \tau
   +\rho ^3-5 \rho ^2+4 \rho \right)\Big),\notag
\end{eqnarray}
\begin{eqnarray}
\widehat{\mathcal{C}}(x;n) & = & 2\Big(16x^6(n+1) (\rho -1)^2
  -8x^5\lambda(\rho -1)\left(   n (4 \rho -8)+3 \rho -8    \right) \notag\\
   &&-8x^4\left(   \lambda ^2 \left(-\rho ^2+10 \rho -12\right)+n^2 \left(2 \rho ^2-4 \rho +2\right)\right.\notag\\
   &&\left.+n\left(\lambda ^2 \left(-2 \rho ^2+12 \rho -12\right)+2 \rho ^3+2 \rho ^2+\left(2
   \rho ^3+2 \rho ^2-8 \rho +4\right) \tau -10 \rho +6\right)\right.\notag\\
   &&\left.+\rho ^3+2 \rho
   ^2+\left(\rho ^3+5 \rho ^2-10 \rho +4\right) \tau -7 \rho +4   \right) \notag\\
   &&+4x^3\lambda\left(      \lambda ^2 (6 \rho -16)+n^2 \left(8 \rho ^2-24 \rho +16\right)\right.\notag\\
   &&\left.+n \left(\lambda ^2 (8
   \rho -16)+4 \rho ^3+12 \rho ^2+\left(4 \rho ^3+8 \rho ^2-48 \rho +32\right) \tau -64
   \rho +48\right)\right.\notag\\
   &&\left.+\rho ^3+9 \rho ^2+\left(\rho ^3+12 \rho ^2-52 \rho +32\right) \tau
   -42 \rho +32    \right)\notag\\   
   &&-2x^2\left( -8 \lambda ^4+\lambda ^2 \left(3 \rho ^2-84 \rho +96\right)\right.\notag\\
   &&+\tau  \left(\lambda ^2
   \left(4 \rho ^2-88 \rho +96\right)-30 \rho ^3+30 \rho ^2\right)\notag\\
   &&\left.+n^2 \left(\lambda ^2
   \left(8 \rho ^2-48 \rho +48\right)-8 \rho ^3+16 \rho ^2+\left(8 \rho ^2-8 \rho
   ^3\right) \tau -8 \rho \right)\right.\notag\\
   &&+n \left(-8 \lambda ^4+\lambda ^2 \left(16 \rho ^2-136
   \rho +144\right)+\tau  \Big(\lambda ^2 \left(8 \rho ^2-96 \rho +96\right)\right. \notag\\
   &&-4 \rho
   ^4-36 \rho ^3+56 \rho ^2-16 \rho \Big)\notag\\
   &&\left.-2 \rho ^4-20 \rho ^3+70 \rho ^2+\left(-2
   \rho ^4-16 \rho ^3+16 \rho ^2\right) \tau ^2-72 \rho +24\right)\notag\\
   &&\left.-18 \rho ^3+60 \rho ^2+\left(-12 \rho ^3-8 \rho ^2+16 \rho \right) \tau ^2-66 \rho +24      \right) \notag\\
   &&-2x\lambda\left(         \tau  \left(\lambda ^2 (24 \rho -64)+10 \rho ^3-28 \rho ^2\right)+\lambda ^2 (28 \rho
   -64)\right.\notag\\
   &&\left.+n^2 \left(\lambda ^2 (16 \rho -32)+8 \rho ^3-24 \rho ^2+\left(8 \rho ^3-16 \rho
   ^2\right) \tau +16 \rho \right)\right.\notag\\
   &&+n \left(\tau  \left(\lambda ^2 (32 \rho -64)+40 \rho
   ^3-96 \rho ^2+32 \rho \right)+\lambda ^2 (48 \rho -96)+24 \rho ^3-96 \rho^2\right.\notag\\
&&\left.+\left(16 \rho ^3-32 \rho ^2\right) \tau ^2+120 \rho -48\right)+\left(4 \rho ^3-32
 \rho \right) \tau ^2\notag\\
 &&\left.+6 \rho ^3-45 \rho ^2+87 \rho -48    \right) \notag \\
 && -\left(32 \lambda ^4+\lambda ^2 \left(45 \rho ^2-66 \rho +48\right)+\tau ^2 \left(\lambda ^2
   \left(16 \rho ^2+32 \rho \right)+64 \rho ^3-56 \rho ^2\right)\right.\notag\\
   &&\left.+\tau  \left(32 \lambda
   ^4+52 \lambda ^2 \rho ^2+84 \rho ^3-132 \rho ^2+48 \rho \right)+n^2 \left(16 \lambda
   ^4+\lambda ^2 \left(16 \rho ^2-16 \rho \right)\right.\right.\notag\\
   &&\left.\left.+\tau  \left(16 \lambda ^2 \rho ^2+8
   \rho ^4-8 \rho ^3\right)+4 \rho ^4 \tau ^2+4 \rho ^4-8 \rho ^3+4 \rho ^2\right)\right.\notag\\
   &&+n
   \Big(48 \lambda ^4+\lambda ^2 \left(66 \rho ^2-96 \rho +48\right)+\tau ^2 \left(32
   \lambda ^2 \rho ^2+28 \rho ^4-16 \rho ^3\right)\notag\\
   &&+\tau  \left(32 \lambda ^4+\lambda ^2
   \left(80 \rho ^2-32 \rho \right)+32 \rho ^4-28 \rho ^3-4 \rho ^2\right)\notag\\
   &&+8 \rho ^4
   \tau ^3+12 \rho ^4-12 \rho ^3-12 \rho ^2+12 \rho \Big)\notag\\
   &&\left.+16 \rho ^3 \tau ^3+36 \rho
   ^3-72 \rho ^2+36 \rho \right)\Big),\notag
\end{eqnarray}
\begin{eqnarray}
\widehat{\mathcal{D}}(x;n)& = &4\Big( 8x^5  (n+1) (\rho -1)^2-8x^4  \lambda (\rho -1)\left(  n (3 \rho -6)+2 \rho -6   \right)\notag\\
   &&+4x^3 \left(   \lambda ^2 \left(2 \rho ^2-20 \rho +24\right)+n^2 \left(4 \rho ^2-8 \rho +4\right)\right.\notag\\
   &&+n
   \big(\lambda ^2 \left(4 \rho ^2-24 \rho +24\right)-4 \rho ^3+20 \rho ^2\notag\\
   &&+\left(-4
   \rho ^3+12 \rho ^2-16 \rho +8\right) \tau -28 \rho +12\big)\notag\\
   &&\left.-\rho ^3+10 \rho
   ^2+\left(-\rho ^3+3 \rho ^2-10 \rho +8\right) \tau -17 \rho +8   \right)\notag\\
   &&+2x^2 \lambda \left(\lambda ^2 (16 \rho -40)+n^2 \left(-12 \rho ^2+36 \rho -24\right)\right.\notag\\
   &&\left.+n \left(\lambda ^2
   (20 \rho -40)+10 \rho ^3-71 \rho ^2+\left(10 \rho ^3-44 \rho ^2+72 \rho -48\right)
   \tau +133 \rho -72\right)\right.\notag\\
   &&\left.+2 \rho ^3-37 \rho ^2+\left(2 \rho ^3-16 \rho ^2+44 \rho
   -48\right) \tau +83 \rho -48 \right)\notag\\
 &&  + 2x \left( 12 \lambda ^4+\lambda ^2 \left(17 \rho ^2-50 \rho +48\right)+\tau  \left(\lambda ^2
   \left(10 \rho ^2-28 \rho +48\right)-5 \rho ^3+33 \rho ^2-28 \rho \right)\right.\notag\\
   &&\left.+n^2
   \left(\lambda ^2 \left(4 \rho ^2-24 \rho +24\right)-4 \rho ^3+8 \rho ^2+\left(4 \rho
   ^2-4 \rho ^3\right) \tau -4 \rho \right)\right.\notag\\
   &&+n \Big(12 \lambda ^4+\lambda ^2 \left(24
   \rho ^2-84 \rho +72\right)+\tau  \big(\lambda ^2 \left(20 \rho ^2-48 \rho
   +48\right)+6 \rho ^4-36 \rho ^3\notag\\
   && +58 \rho ^2-28 \rho \big)+3 \rho^4 - 28 \rho^3 + 71
   \rho^2 + \left (3 \rho^4 - 8 \rho^3 + 8 \rho^2 \right)\tau^2 - 70 \rho + 24 \Big)  \notag\\
   &&+\lambda\left( \tau  \left(\lambda ^2 (8 \rho -32)+10 \rho ^3-66 \rho ^2+56 \rho \right)+\lambda ^2 (2
   \rho -32)\right.\notag\\
   &&\left.+n^2 \left(\lambda ^2 (8 \rho -16)+4 \rho ^3-12 \rho ^2+\left(4 \rho ^3-8
   \rho ^2\right) \tau +8 \rho \right)\right.\notag\\
   &&\left.+n \left(\tau  \left(\lambda ^2 (16 \rho -32)+25
   \rho ^3-88 \rho ^2+56 \rho \right)+\lambda ^2 (14 \rho -48)+17 \rho ^3-85 \rho
   ^2\right.\right.\notag\\
   &&\left.\left.+\left(8 \rho ^3-16 \rho ^2\right) \tau ^2+116 \rho -48\right)+6 \rho ^3-60 \rho
   ^2+\left(4 \rho ^3-16 \rho ^2\right) \tau ^2+102 \rho -48\right)\Big), \notag
 \end{eqnarray}
 \begin{eqnarray}
\widehat{\mathcal{E}}(x;n)& = & 4(n+1)\Big(  
-8x^4 (n+1)^2 (\rho -1)^2\notag\\
&&+4x^3\lambda   (\rho -1) \left(n^2 (4 \rho -8)+n (5 \rho -16)+2 \rho -10)\right)\notag\\
   &&+4x^2 \left(\lambda ^2 (14-8 \rho )+n^2 \left(\lambda ^2 \left(2 \rho ^2-12 \rho +12\right)-2 \rho
   ^3+4 \rho ^2+\left(2 \rho ^2-2 \rho ^3\right) \tau -2 \rho \right)\right.\notag\\
   &&\left.+n \left(\lambda
   ^2 \left(\rho ^2-18 \rho +24\right)-\rho ^3+6 \rho ^2+\left(-\rho ^3-5 \rho ^2+6
   \rho \right) \tau -9 \rho +4\right)-\rho ^3+10 \rho ^2\right.\notag\\
   &&\left.+\left(-\rho ^3+3 \rho ^2-10
   \rho +8\right) \tau -17 \rho +8\right)\notag\\
   &&+2x\lambda  \left(-12 \lambda ^2+n^2 \left(\lambda ^2 (8 \rho -16)+8 \rho ^3-8 \rho ^2 \tau -12 \rho ^2+8
   \rho \right)\right.\notag\\
   &&+n \left(\lambda ^2 (10 \rho -32)-\rho ^3-11 \rho ^2+\left(-\rho ^3+8
   \rho ^2-24 \rho \right) \tau +28 \rho -16\right)\notag\\
   &&\left.+\left(-10 \rho ^2+24 \rho-32\right) \tau -17 \rho ^2+49 \rho -32)\right)\notag\\
   &&-\left(  \tau  \left(\lambda ^2 (32-8 \rho )-10 \rho ^3+66 \rho ^2-56 \rho \right)+\lambda ^2
   (32-2 \rho )\right.\notag\\
   &&\left.+n^2 \left(8 \lambda ^4+\lambda ^2 \left(8 \rho ^2-8 \rho \right)+\tau 
   \left(8 \lambda ^2 \rho ^2+4 \rho ^4-4 \rho ^3\right)+2 \rho ^4 \tau ^2+2 \rho ^4-4
   \rho ^3+2 \rho ^2\right)\right.\notag\\
   &&\left.+n \left(16 \lambda ^4+\lambda ^2 \left(19 \rho ^2-20 \rho
   +16\right)+\tau  \left(\lambda ^2 \left(4 \rho ^2+24 \rho \right)-4 \rho ^4+38 \rho
   ^3-34 \rho ^2\right)\right.\right.\notag\\
   &&\left.-2 \rho ^4+26 \rho ^3-46 \rho ^2+\left(12 \rho ^3-2 \rho
   ^4\right) \tau ^2+22 \rho \right)\notag\\
   &&\left.-6 \rho ^3+60 \rho ^2+\left(16 \rho ^2-4 \rho
   ^3\right) \tau ^2-102 \rho +48 \right)\Big).\notag
   \end{eqnarray}
   
\subsubsection{Results for the classical Hermite family}
\vspace{0.25 cm}

\noindent  If  $\tau=0$, $\lambda =0$, and $\rho=1$, we recover the classical Hermite family \cite{Maroni-1994}.
\vspace{0.25 cm}

\noindent  {\bf Coefficients of the four structure relations, for $n\geq 0$}
\begin{eqnarray}
&& G_{0,1}=0,\quad G_{1,1}=0,\quad H_{1}= n+1,\quad \Phi=1,\quad M_{0,1}=0. \notag\\
&& G_{0,2}=0 ,\quad G_{1,2}= 0,\quad H_{2}=2 (n+1) x ,\quad \Phi^2=1,\quad M_{1,2}=0,\quad M_{0,2}=2(n+1). \notag \\ 
&&G_{0,3}=0,\quad G_{1,3}= 0,\quad H_{3}=4 (n+1) x^2+2 (n+1),\quad \Phi^3=1,\notag\\
&& M_{2,3}=0 ,\quad M_{1,3}= 2 (n+1),\quad M_{0,3}= 4 (n+1) x. \notag\\
&&G_{0,4}=0,\quad G_{1,4}=0,\quad H_{4}= 8 (n+1) x^3+12 (n+1) x,\quad \Phi^4=1,\notag\\
&& M_{3,4}=0 ,\quad M_{2,4}=2 ( n+1) ,\quad M_{1,4}=4 (n+1) x,\quad 
M_{0,4}= 8 (n+1) x^2+8 (n+1). \notag
\end{eqnarray}
\noindent  {\bf The four structure relations, for $n\geq 0$}
\begin{eqnarray}
&& (n+1)P_{n}(x)= P'_{n+1}(x),\notag\\
&& 2 (n+1) x P_{n}(x)= P''_{n+1}(x)+2(n+1)P_{n+1}(x),\notag\\
&& 2 (n+1)\left( 2x^2+1\right)P_{n}(x)=P^{(3)}_{n+1}(x)+2 (n+1)P'_{n+1}(x)+4 (n+1) xP_{n+1}(x),\notag\\
&& 4(n+1) x\left(2x^2+3\right)P_{n}(x)=P^{(4)}_{n+1}(x)+2 ( n+1) P''_{n+1}(x)+4 (n+1) xP'_{n+1}(x)+\notag\\
&&8 (n+1) (x^2+1)P_{n+1}(x).\notag
\end{eqnarray}
\noindent {\bf Determinants}
\begin{eqnarray}
&& \Delta_1=0 ,\quad  \Delta_2=0 ,\quad  \Delta_3=0 ,\quad  \Delta_4=0.\notag
\end{eqnarray}
\noindent {\bf Coefficients of the fourth-order differential equation as Laguerre-Hahn sequence, for $n\geq 0$}
\begin{eqnarray}
&& \widehat{\mathcal{A}}(x;n)=0, \quad
\widehat{\mathcal{B}}(x;n)=  0, \quad
\widehat{\mathcal{C}}(x;n)= 0, \quad \widehat{\mathcal{D}}(x;n)= 0,\quad 
\widehat{\mathcal{E}}(x;n)= 0.\notag \\
\notag 
\end{eqnarray}
\noindent  {\bf The second-order differential equation as semiclassical sequence}
\[
\mathcal{C}_2(x;n)P_{n+1}''(x) + \mathcal{D}_2(x;n)P_{n+1}'(x) + \mathcal{E}_2(x;n)P_{n+1}(x) = 0,\ n\geq 0.
\]
\noindent  {\bf Greatest common factor between} $\mathcal{C}_2$, $\mathcal{D}_2$, and $\mathcal{E}_2$:
$$c(x;n)=n+1, \ n\geq 0.$$
\noindent  {\bf Reduced coefficients of the second-order differential equation}
$$\widehat{\mathcal{C}}_2=\mathcal{C}_2/c,\quad \widehat{\mathcal{D}}_2=\mathcal{D}_2/c,\quad 
\widehat{\mathcal{E}}_2=\mathcal{E}_2/c.$$
$$
\widehat{\mathcal{C}}_2(x;n) = 1 ,\quad
\widehat{\mathcal{D}}_2(x;n) = -2x ,\quad
\widehat{\mathcal{E}}_2(x;n) =  2(n+1).
$$

\vspace{0.25cm}
\noindent  {\bf The third-order differential equation as semiclassical sequence}
\[
\mathcal{B}_3(x;n)P_{n+1}'''(x) + \mathcal{C}_3(x;n)P_{n+1}''(x) + \mathcal{D}_3(x;n)P_{n+1}'(x) + \mathcal{E}_3(x;n)P_{n+1}(x) = 0,\ n\geq 0.
\]
\noindent  {\bf Greatest common factor between} $\mathcal{B}_3$, $\mathcal{C}_3$, $\mathcal{D}_3$, and $\mathcal{E}_3$:
$$c(x;n)=n+1,\ n\geq 0.$$
\noindent  {\bf Reduced coefficients of the third-order differential equation}
$$\widehat{\mathcal{B}}_3=\mathcal{B}_3/c,\quad \widehat{\mathcal{C}}_3=\mathcal{C}_3/c,\quad \widehat{\mathcal{D}}_3=\mathcal{D}_3/c,\quad \widehat{\mathcal{E}}_3=\mathcal{E}_3/c.$$
$$
\widehat{\mathcal{B}}_3(x;n) = 1,\quad
\widehat{\mathcal{C}}_3(x;n) =0 ,\quad
\widehat{\mathcal{D}}_3(x;n) = -4x^2+2n ,\quad
\widehat{\mathcal{E}}_3(x;n) =4x(n+1) .
$$

\vspace{0.25cm}
\noindent  {\bf The fourth-order differential equation as semiclassical form}
\begin{eqnarray}
\mathcal{A}_4(x;n)P_{n+1}^{(4)}(x) + \mathcal{B}_4(x;n)P_{n+1}'''(x) + \mathcal{C}_4(x;n)P_{n+1}''(x) + \widehat{\mathcal{D}}_4(x;n)P_{n+1}'(x) &&\notag\\
+ \mathcal{E}_4(x;n)P_{n+1}(x) = 0,\ n\geq 0. &&\notag
\end{eqnarray}
\noindent  {\bf Greatest common factor between} $\mathcal{A}_4$, $\mathcal{B}_4$, $\mathcal{C}_4$, $\mathcal{D}_4$, and $\mathcal{E}_4$: $$c(x;n)=n+1,\ n\geq 0.$$
\noindent  {\bf Reduced coefficients of the fourth-order differential equation}
$$
\widehat{\mathcal{A}}_4=\mathcal{A}_4/c,\quad \widehat{\mathcal{B}}_4=\mathcal{B}_4/c,\quad \widehat{\mathcal{C}}_4=\mathcal{C}_4/c,\quad \widehat{\mathcal{D}}_4=\mathcal{D}_4/c,\quad \widehat{\mathcal{E}}_4=\mathcal{E}_4/c.
$$
\begin{align*}
&\widehat{\mathcal{A}}_4(x;n) = 1,\quad
\widehat{\mathcal{B}}_4(x;n) =0 ,\quad
\widehat{\mathcal{C}}_4(x;n) = 2(n+1),\\
&\widehat{\mathcal{D}}_4(x;n) =4x(-2x^2+n-2),\quad
\widehat{\mathcal{E}}_4(x;n) = 8(x^2+1)(n+1),\ n\geq 0.
\end{align*}
This fourth-order differential equation differs from the one presented in \cite{Loureiro-2006}, which guarantees the characterization of the classical sequence.

\subsection{Case 2 analogous to Hermite}\label{Section5_H_Case2}

\vspace{0.25cm}

\noindent \textbf{Regularity condition}
\[
\lambda, \rho \in\mathbb{C}, \quad \rho\neq 0.
\]
\noindent {\bf Recurrence coefficients} 
$$
\beta_{0}=\lambda,\quad \beta_{n+1}=0,\ n\geq 0;\qquad 
\gamma_{1}=\frac{\rho}{2},\quad \gamma_{n+1}=\frac{n}{2},\ n\geq 1.
$$
{\bf Coefficients of the Stieltjes equation} 
$$ \Phi(x)=1,\quad B(x)=2x^2-2\lambda x +1-\rho,\quad  C(x)=2x,\quad D(x)=0.$$
{\bf Coefficients of the Laguerre-Hahn structure relation} 
$$C_{n+1}(x)=-2x,\quad  D_{n+1}(x)=-2,\quad n\geq 0.$$
{\bf Relation to the classical Hermite form ${\cal H}$} \cite[Proposition 4.2]{Mohamed-Imed-2025}
$$
u_{0}^{(1)}={\cal H}.
$$
\noindent {\bf First structure relation}
\begin{eqnarray}
G_{0,1}(x;n)P^{(1)}_{n-1}(x)+G_{1,1}(x;n)P^{(1)}_{n}(x)+H_{1}(x;n)P_{n}(x)=&&\notag\\
\Phi(x) P'_{n+1}(x)+M_{0,1}(x;n)P_{n+1}(x),\ n\geq 0.\notag&&
\end{eqnarray}
 \begin{eqnarray}
&&G_{0,1}(x;n)=0,\quad 
G_{1,1}(x;n)=2x^2-2x\lambda  +1-\rho;\notag\\
&&  H_{1}(x;0)=\rho,\quad  H_{1}(x;n)=n,\ n\geq 1;\notag\\
&& \Phi(x)=1,\quad M_{0,1}(x;n)= 2x.\notag
\end{eqnarray}

\noindent {\bf Second structure relation}
\begin{eqnarray}
G_{0,2}(x;n)P^{(1)}_{n-1}(x)+ G_{1,2}(x;n)P^{(1)}_{n}(x)+ H_{2}(x;n)P_{n}(x)= &&\notag\\ 
\Phi^2(x)P''_{n+1}(x)+M_{1,2}(x;n)P'_{n+1}(x)+M_{0,2}(x;n)P_{n+1}(x),\ n\geq 0.&&\notag
\end{eqnarray}
 \begin{eqnarray}
&& G_{0,2}(x;0)= 2\rho\left(2 x^2-2x \lambda  -(\rho -1) \right),\notag\\
&&  G_{0,2}(x;n)  = 2n\left( 2 x^2-2x \lambda   - (\rho -1)\right),\ n\geq 1;\notag\\
&& G_{1,2}(x;n)  =2(2 x- \lambda),\quad H_{2}(x;n) =  0,\quad \Phi^2(x) = 1,\notag\\
&&  M_{1,2}(x;n) = 2x,\quad M_{0,2}(x;n)  =  2(n+1).\notag 
\end{eqnarray}

\noindent {\bf Third structure relation}
\begin{eqnarray}
G_{0,3}(x;n)P^{(1)}_{n-1}(x)+G_{1,3}(x;n)P^{(1)}_{n}(x)+H_{3}(x;n)P_{n}(x)=&&\notag\\
\Phi^3(x)P^{(3)}_{n+1}(x)+M_{2,3}(x;n)P''_{n+1}(x)+M_{1,3}(x;n)P'_{n+1}(x)+M_{0,3}(x;n)P_{n+1}(x),\ n\geq 0.&&\notag
\end{eqnarray}
 \begin{eqnarray}
G_{0,3}(x;0)&= &2\rho\left(4 x^3-4 x^2\lambda  -2 x(\rho -4)   -3 \lambda \right),\notag\\  
G_{0,3}(x;n)&= &2n\left( 4 x^3-4 x^2\lambda   -2 x(\rho -4) -3 \lambda  \right),\ n\geq 1;\notag\\
G_{1,3}(x;n)&= & 4\left(-2 n x^2+2x \lambda  n + (n (\rho -1)+1)\right),\quad H_{3}(x;n)= 0,\quad \Phi^3(x)=1,\notag\\
M_{2,3}(x;n)&= &2x,\quad M_{1,3}(x;n)=2 (2 + n),\quad  M_{0,3}(x;n)= 0.\notag
\end{eqnarray}

\noindent {\bf Fourth structure relation}
\begin{eqnarray}
G_{0,4}(x;n) P^{(1)}_{n-1}(x)+G_{1,4}(x;n) P^{(1)}_{n}(x)+H_{4}(x;n)P_{n}(x)=&&\notag\\
\Phi^4(x)P^{(4)}_{n+1}(x)+M_{3,4}(x;n) P^{(3)}_{n+1}(x)+M_{2,4}(x;n) P''_{n+1}(x)&&\notag\\
+M_{1,4}(x;n) P'_{n+1}(x)+M_{0,4}(x;n) P_{n+1}(x),\ n\geq 0.&&\notag
\end{eqnarray}
 \begin{eqnarray}
G_{0,4}(x;0)  & = & 4\rho\left(4 x^4-4 x^3\lambda   -4 x^2(\rho -7)   -7 x\lambda  -(\rho -5)\right) ,\notag\\
G_{0,4}(x;n) & = &4n\left(
4 x^4-4 x^3\lambda   -2x^2  (n+\rho -7)+x\lambda   (2n-7) + (n (\rho -1)-\rho +5)\right),n\geq 1;\notag\\
G_{1,4}(x;n) & = & 4n\left(
 -4 x^3+4 \lambda   x^2+2 x(\rho -6) +5\lambda  \right),\quad H_{4}(x;n)=0,\quad \ \Phi^4(x)=1,\notag\\
M_{3,4}(x;n)& = &2x ,\quad 
M_{2,4}(x;n)= 2 (3 + n) ,\quad 
M_{1,4}(x;n)= 0 ,\quad 
M_{0,4}(x;n) = 0.\notag
\end{eqnarray}

\noindent  {\bf The fourth-order linear differential equation}
\begin{eqnarray}
\mathcal{A}(x;n)P^{(4)}_{n+1}(x)
+\mathcal{B}(x;n)P^{(3)}_{n+1}(x)
+\mathcal{C}(x;n)P''_{n+1}(x)
+\mathcal{D}(x;n)P'_{n+1}(x)&&\notag\\
+\mathcal{E}(x;n)P_{n+1}(x)=0, \ n\geq 0.&&\notag
\end{eqnarray}
Greatest common factor between $\mathcal{A}$, $\mathcal{B}$, $\mathcal{C}$, $\mathcal{D}$, and $\mathcal{E}$: 
\begin{eqnarray}
 c(x;0)=4 \rho ^2,\quad  c(x;n)= 4n^2,\quad n\geq 1. \notag
\end{eqnarray}
$$\widehat{\mathcal{A}}=\mathcal{A}/c ,\quad \widehat{\mathcal{B}}=\mathcal{B}/c ,\quad \widehat{\mathcal{C}}=\mathcal{C}/c,\quad \widehat{\mathcal{D}}=\mathcal{D}/c,\quad \widehat{\mathcal{E}}=\mathcal{E}/c.$$
\begin{eqnarray}
\widehat{\mathcal{A}}(x;n) & = & -8 x^4(n+1) +4x^3 \lambda  (4 n+3) -4 x^2
   \left(\lambda ^2+2  n(\lambda ^2- \rho +1)-\rho +3\right)\notag\\
   &&-2 x\lambda   (4 n (\rho-1)+\rho -5)-2 n (\rho ^2+2 \rho -1)-3 \lambda ^2-2 \rho +2, \notag \\
   &&\notag\\
\widehat{\mathcal{B}}(x;n) & = &  32 x^3(n+1) -12 x^2\lambda  (4 n+3) +8 x
   \left(\lambda ^2+2  n(\lambda ^2- \rho +1)-\rho +3\right)\notag\\
   &&+2 \lambda  (4 n (\rho -1)+\rho -5),\notag \\
      &&\notag\\
\widehat{\mathcal{C}}(x;n) & = & 32 x^6(n+1) -16 x^5\lambda  (4 n+3) -16 x^4 \left(2 n^2+n(-2 \lambda ^2 +2  \rho)
   +\rho -3-\lambda ^2\right)\notag\\
   &&+8x^3 \lambda   \left(8 n^2+4 n (\rho +1)+\rho -3\right)\notag\\
   &&-4 x^2\left(8  n^2(\lambda ^2-\rho +1)+ 2n(4 \lambda ^2-  \rho ^2-2 \rho +15 )-\lambda ^2-6 \rho +22\right)\notag\\
   &&-4 x\lambda   \left(8 n^2( \rho -1)+ n (8\rho -32 )+2 \rho -17\right)\notag \\
   &&-2 \left(9 \lambda ^2+4 n^2 (\rho ^2-2 \rho +1)+2n(9 \lambda ^2 +2 \rho ^2+2 
   \rho -8 )+4 (\rho +1)\right),\notag\\
 &&\notag\\
\widehat{\mathcal{D}}(x;n) & = & 32 x^5(n+1) -32 x^4\lambda  (3 n+2) +16 x^3 \left(2 \lambda ^2+4 n(n+ \lambda ^2 - \rho +2 )-\rho +7\right)\notag\\
   &&-8x^2 \lambda \left(12 n^2+ n(- 10\rho +27) -2 \rho +21\right)\notag\\
   &&+8 x \left(7 \lambda ^2+4  n^2(\lambda^2-\rho +1)+ n(4 \lambda ^2+3 \rho ^2-20 \rho +21)-\rho +5\right)\notag\\
   &&+4\lambda  (n+2) (4 n (\rho -1)+\rho -5),\notag\\
   &&\notag\\
\widehat{\mathcal{E}}(x;n) & = & -32 x^4(n+1)^3 +16 x^3\lambda  (n+1) \left(4 n^2+5 n+2\right) \notag\\
   &&- 16 x^2(n+1) \left(2n^2( \lambda ^2 - \rho +1 )+n(\lambda ^2 - \rho +11 )-\rho +7\right)\notag\\
   &&-8x\lambda  (n+1) \left(4 n^2 (\rho -1)-n (\rho +19)-7\right) \notag\\
   &&- 4 (n+1) \left(2 n^2 (\rho ^2-2  \rho +2 )+n(15 \lambda ^2 -2 \rho ^2+14  \rho-12 )-2 \rho +10\right).\notag\\
\notag 
\end{eqnarray}

\section{Results for a semiclassical family of class~1}\label{Section7}

In this section, we list the results obtained corresponding to a semiclassical family of class 1, almost symmetric \cite{Maroni-Mejri-2011}, namely the second-order, third-order, and fourth-order differential equations satisfied by this polynomial sequence. Structure relations are available in the software.  Our results for the second-order differential equation coincide with those reported in \cite{Maroni-Mejri-2011}, thereby confirming the effectiveness of our implementation.

The following characteristic elements of this sequence, using as input data of the algorithm, were given in 
\cite{Maroni-Mejri-2011}.

\vspace{0.25cm}

\noindent \textbf{Regularity conditions} 
$$
\alpha, \beta \in \mathbb{C},  \quad \alpha \neq -(n+1), \quad \beta \neq -(n+1), \quad \alpha + \beta \neq -(n+1), \quad n \geq 0.
$$
\noindent \textbf{Recurrence coefficients}
\begin{eqnarray}
\beta_{n} &=& (-1)^n, \quad n \geq 0, \notag\\
\gamma_{2n+1}& =& -\frac{(n + \alpha+1)(n + \alpha + \beta+1)}{(2n + \alpha + \beta + 1)(2n + \alpha + \beta+2)}, \quad  n \geq 0,\notag\\
\gamma_{2n+2} &=& -\frac{(n +1)(n +  \beta+1)}{(2n + \alpha + \beta + 2)(2n + \alpha + \beta+3)}, \quad  n \geq 0. \notag
\end{eqnarray}
\noindent \textbf{Coefficients of the Stieltjes equation}
$$
\Phi(x) = x(x^2-1),~~ B(x) =0,~~ C(x) = (2\alpha + 2\beta+1)x^2-x-2\beta, ~~
D(x) = 2(\alpha + \beta+1)(x+1). 
$$ 
\noindent \textbf{Coefficients of the Laguerre-Hahn structure relation}
\begin{align*}
C_{n+1}(x) =& (2n+2\alpha + 2\beta+1)x^2+(-1)^{n+1}x-2\beta-2n+(2\alpha+1)\left((-1)^n-1\right),n \geq 0, \\
D_{n+1}(x) =& 2(n+\alpha + \beta+1)(x+(-1)^n), \ n \geq 0.
\end{align*}

\noindent Next, we present the list of the results obtained using the software. 

\noindent {\bf The second-order differential equation as a semiclassical sequence}
$$
\mathcal{C}_2(x;n)P''_{n+1}(x)
+\mathcal{D}_2(x;n)P'_{n+1}(x)
+\mathcal{E}_2(x;n)P_{n+1}(x)=0, \ n\geq 0.
$$
Greatest common factor between $\mathcal{C}_2$, $\mathcal{D}_2$, and $\mathcal{E}_2$: 
\begin{eqnarray}
 c_2(x;2n)& = & \frac{2 (\alpha +n+1) (\alpha +\beta +n+1)}{\alpha +\beta +2 n+1}x(x^2-1), \ n\geq 0,\notag\\
c_2(x;2n+1) & = & \frac{2 (n+1)  (\beta +n+1)}{\alpha +\beta +2 n+2}(x+1) x (x^2-1), \ n\geq 0.\notag
\end{eqnarray}
\noindent  {\bf Reduced coefficients of the second-order differential equation}
$$\widehat{\mathcal{C}}_2=\mathcal{C}_2/c_2, \quad
\widehat{\mathcal{D}}_2=\mathcal{D}_2/c_2,\quad \widehat{\mathcal{E}}_2=\mathcal{E}_2/c_2.$$
\begin{eqnarray}
\widehat{\mathcal{C}}_2(x;2n)& = & (x-1) x (x^2-1),\notag \\
\widehat{\mathcal{D}}_2(x;2n)& = & (x-1) \left(x^2 (2 \alpha +2 \beta +3)-2 x-2 \beta -1\right),\notag \\
\widehat{\mathcal{E}}_2(x;2n)& = & -(2 n+1) x^2 (2 \alpha +2
   \beta +2 n+3) \notag \\
&& + x \left(4 n^2+4 n (\alpha +\beta +2)+2\right)+2 \beta +1.\notag \\ \notag\\
\widehat{\mathcal{C}}_2(x;2n+1)& = &x ( x^2-1),\notag \\
\widehat{\mathcal{D}}_2(x;2n+1)& = & x^2 (2 \alpha +2 \beta +3)-2 \beta -1,\notag \\
\widehat{\mathcal{E}}_2(x;2n+1)& = & -4 x(n+1) (\alpha +\beta +n+2).\notag 
\end{eqnarray}  
\noindent {\bf The third-order differential equation as a semiclassical sequence}
$$
\mathcal{B}_3(x;n)P^{(3)}_{n+1}(x)
+\mathcal{C}_3(x;n)P''_{n+1}(x)
+\mathcal{D}_3(x;n)P'_{n+1}(x)
+\mathcal{E}_3(x;n)P_{n+1}(x)=0, \ n\geq 0.
$$
Greatest common factor between $\mathcal{B}_3$, $\mathcal{C}_3$, $\mathcal{D}_3$, and $\mathcal{E}_3$: 
\begin{eqnarray}
 c_3(x;2n) & = & \frac{2  (\alpha +n+1) (\alpha +\beta +n+1)}{\alpha +\beta +2 n+1}x (x^2-1), \ n\geq 0, \notag\\
c_3(x;2n+1) & = & \frac{2 (n+1) (\beta +n+1)}{\alpha +\beta +2 n+2}(x+1) x (x^2-1), \ n\geq 0. \notag
\end{eqnarray}
\noindent  {\bf Reduced coefficients of the third-order differential equation}
$$\widehat{\mathcal{B}}_3=\mathcal{B}_3/c_3,\quad \widehat{\mathcal{C}}_3=\mathcal{C}_3/c_3,\quad 
\widehat{\mathcal{D}}_3=\mathcal{D}_3/c_3,\quad \widehat{\mathcal{E}}_3=\mathcal{E}_3/c_3.$$
\begin{eqnarray}
\widehat{\mathcal{B}}_3(x;2n)& = &(x-1) x^2 (x^2-1)^2 ,\notag \\
\widehat{\mathcal{C}}_3(x;2n)& = & -(-1 + x) x (x^2-1)\left(  2 x^2(n-4)+x -2 n-2 \alpha+1\right),\notag \\
\widehat{\mathcal{D}}_3(x;2n)& = & -(-1 + x)\Big( 
x^4 \left(4 \alpha ^2+8 \alpha  \beta +4 \beta ^2+4 n^2+n (8 \alpha +8 \beta +14)-9\right)\notag\\
&&-4x^3 (\alpha +\beta +n-1)\notag\\
&&-2 x^2 \left(2 \alpha ^2+\alpha  (6 \beta +3)+4 \beta ^2+2 \beta +2 n^2+n (4 \alpha +6 \beta +8)-2\right)\notag\\
&&+(2 \beta +1) (2 \alpha +2 \beta +2 n+1) +4 x (\alpha +\beta +n) \Big),\notag \\
\widehat{\mathcal{E}}_3(x;2n)& = &(2 n+1) x^4 (2 \alpha +2 \beta +2 n-3) (2
   \alpha +2 \beta +2 n+3)\notag\\
&& -4 x^3 (\alpha +\beta +n-1) \left(2 n^2+n (2 \alpha +2 \beta +4)+1\right)\notag\\
   &&-2 x^2 (2\alpha +2 \beta +2 n-1) \left(\alpha +2 \beta +2 n^2+n (2 \alpha +2 \beta +4)+2\right)\notag\\
   &&+4 x (\alpha +\beta +n) \left(2 n^2+n (2 \alpha +2 \beta +4)+1\right)\notag \\
   &&+(2
   \beta +1) (2 \alpha +2 \beta +2 n+1),\notag \\  \notag \\ 
\widehat{\mathcal{B}}_3(x;2n+1)& = & x^2 (x^2-1)^2,\notag \\
\widehat{\mathcal{C}}_3(x;2n+1)& = & -x (x^2-1)\left((2 n-7) x^2-2 n+1\right),\notag \\
\widehat{\mathcal{D}}_3(x;2n+1)& = & x^4 \left(-4 \alpha ^2+\alpha  (-8 \beta -4)-4 \beta ^2-4 \beta -4 n^2+n (-8 \alpha -8 \beta
   -18)+1\right)\notag\\
   &&+4 x^2 \left(\alpha  (2 \beta +1)+2 \beta ^2+2 \beta +n^2+n (2 \alpha +3 \beta +5)+1\right)\notag\\
   &&-(2 \beta +1) (2 \beta +2 n+1),\notag \\
\widehat{\mathcal{E}}_3(x;2n+1)& = & 8 (n+1) x (\alpha +\beta +n+2)\left(-\beta +x^2 (\alpha +\beta +n-1)-n\right).\notag 
\end{eqnarray}  

\noindent {\bf The fourth-order differential equation as a semiclassical sequence}
\begin{eqnarray}
\mathcal{A}_4(x;n)P^{(4)}_{n+1}(x)
+\mathcal{B}_4(x;n)P^{(3)}_{n+1}(x)
+\mathcal{C}_4(x;n)P''_{n+1}(x)
+\mathcal{D}_4(x;n)P'_{n+1}(x)&&\notag\\
+\mathcal{E}_4(x;n)P_{n+1}(x)=0, \ n\geq 0.&&\notag
\end{eqnarray}
\noindent {\bf Greatest common factor between} $\mathcal{A}_4$, $\mathcal{B}_4$, $\mathcal{C}_4$, $\mathcal{D}_4$, and $\mathcal{E}_4$: 
\begin{eqnarray}
 c_4(x;2n) & = & \frac{2  (\alpha +n+1) (\alpha +\beta +n+1)}{\alpha +\beta +2 n+1}x (x^2-1), \ n\geq 0, \notag\\ 
c_4(x;2n+1) & = & \frac{2 (n+1)  (\beta +n+1)}{\alpha +\beta +2 n+2}(x+1) x (x^2-1), \ n\geq 0.  \notag
\end{eqnarray}
\noindent  {\bf Reduced coefficients of the fourth-order differential equation}
$$\widehat{\mathcal{A}}_4=\mathcal{A}_4/c_4,\quad \widehat{\mathcal{B}}_4=\mathcal{B}_4/c_4,\quad \widehat{\mathcal{C}}_4=\mathcal{C}_4/c_4, 
\ \widehat{\mathcal{D}}_4=\mathcal{D}_4/c_4,\quad \widehat{\mathcal{E}}_4=\mathcal{E}_4/c_4.$$
\begin{eqnarray}
\widehat{\mathcal{A}}_4(x;2n)& = & (-1 + x)^4 x^3 (1 + x)^3,\notag \\
\widehat{\mathcal{B}}_4(x;2n)& = &-(-1 + x)^3 x^2 (1 + x)^2\Big((2 n-17) x^2+x-2 (\alpha +n-2)
\Big) ,\notag\\ 
\widehat{\mathcal{C}}_4(x;2n)& = & -(-1 + x)^2 x (1 + x)\Big(6 x^4(5 n-12) +12 x^3\notag\\
&&
+x^2 \left(4 \alpha ^2+\alpha  (4 \beta -10)+4 \beta +4 n^2+n (8 \alpha +4 \beta -28)+43\right)\notag\\
&&-6 x-4 \alpha ^2+\alpha  (-4 \beta -2)-4 \beta +n (-8 \alpha -4 \beta -6)-5
\Big),\notag \\
\widehat{\mathcal{D}}_4(x;2n)& = &(-1 + x)\Big(
x^6 \left(8 \alpha ^3+\alpha ^2 (24 \beta -12)+\alpha  \left(24 \beta ^2-24 \beta -2\right)+\right.\notag\\
&&8\beta ^3-12 \beta ^2-2 \beta +8 n^3+n^2 (24 \alpha +24 \beta -12)\notag\\
&&\left.+n \left(24 \alpha^2+\alpha  (48 \beta -24)+24 \beta ^2-24 \beta -116\right)+51\right)\notag\\
&&-6 x^5 \left(2 \alpha ^2+\alpha  (4 \beta -1)+2 \beta ^2-\beta +2 n^2+n (4 \alpha +4 \beta -1)+5\right)\notag\\
    &&+x^4 \left(-16 \alpha ^3+\alpha ^2 (-56 \beta -8)+\alpha  \left(2-64 \beta ^2\right)-24 \beta ^3+8
   \beta ^2-4 \beta \right.\notag\\
&&-16 n^3+n^2 (-48 \alpha -56 \beta -8)\notag\\
&&\left.+n \left(-48 \alpha ^2+\alpha 
   (-112 \beta -16)-64 \beta ^2+146\right)-57\right)\notag\\
     &&+4 x^3 \left(5 \alpha ^2+\alpha  (11 \beta +1)+6 \beta ^2+2 \beta +5 n^2+n (10 \alpha +11 \beta
   +1)+8\right)\notag\\
   &&+x^2 \left(8 \alpha ^3+\alpha ^2 (40 \beta +24)+\alpha  \left(56 \beta ^2+44 \beta +14\right)\right.\notag\\
   &&+24\beta ^3+20 \beta ^2+18 \beta +8 n^3+n^2 (24 \alpha +40 \beta +24)\notag\\
   &&\left.+n \left(24 \alpha^2+\alpha  (80 \beta +48)+56 \beta ^2+44 \beta -24\right)+17\right)\notag\\
   &&-2 x \left(4 \alpha ^2+\alpha  (10 \beta +5)+6 \beta ^2+7 \beta +4 n^2+n (8 \alpha +10 \beta +5)+5\right)\notag\\ 
  && -(2 \beta +1) \left(4 \alpha ^2+\alpha  (8 \beta +6)+4 \beta ^2+6 \beta +4 n^2+n (8 \alpha +8 \beta +6)+3\right)
\Big) ,\notag \\
\widehat{\mathcal{E}}_4(x;2n)& = & -x^6(2 n+1) (2 n+2 \alpha +2 \beta +3) \left(4 \alpha ^2+\alpha  (8 \beta -12)+4 \beta ^2-12 \beta +4 n^2\right.\notag\\
&&\left.+n (8 \alpha +8 \beta-12)+17\right) \notag\\
&&+2 x^5
   \Big(6 \alpha ^2+\alpha  (12 \beta -3)+6 \beta ^2-3 \beta +8 n^4+n^3 (24 \alpha +24 \beta)\notag\\
   &&+n^2 \left(24 \alpha ^2+48 \alpha  \beta +24 \beta ^2-2\right)\notag\\
   &&+n \left(8 \alpha
   ^3+24 \alpha ^2 \beta +\alpha  \left(24 \beta ^2+4\right)+8 \beta ^3+4 \beta
   +45\right)+15\Big)\notag\\
   &&+x^4\left(16 \alpha ^3+\alpha ^2 (56 \beta +8)+\alpha  \left(64 \beta ^2-2\right)+24 \beta ^3-8
   \beta ^2+4 \beta +32 n^4\right.\notag\\
   &&+n^3 (96 \alpha +96 \beta +16)\notag\\
   &&+n^2 \left(96 \alpha ^2+\alpha
    (192 \beta +48)+96 \beta ^2+56 \beta \right)+n \left(32 \alpha ^3+\alpha ^2 (96
   \beta +48)\right.\notag\\
   &&\left.\left.+\alpha  \left(96 \beta ^2+112 \beta +8\right)+32 \beta ^3+64 \beta ^2-8
   \beta +94\right)+57\right) \notag\\
   &&-4 x^3\left(5 \alpha ^2+\alpha  (11 \beta +1)+6 \beta ^2+2 \beta +8 n^4+n^3 (24 \alpha +24 \beta
   +12)\right.\notag\\
   &&\left.+n^2 \left(24 \alpha ^2+\alpha  (48 \beta +24)+24 \beta ^2+24 \beta +9\right)\right.\notag\\
   &&+n
   \left(8 \alpha ^3+\alpha ^2 (24 \beta +12)+\alpha  \left(24 \beta ^2+24 \beta
   +14\right)\right.\notag\\
   &&\left.\left.+8 \beta ^3+12 \beta ^2+15 \beta +25\right)+8\right) \notag\\
   &&+ x^2\left(-8 \alpha ^3+\alpha ^2 (-40 \beta -24)+\alpha  \left(-56 \beta ^2-44 \beta
   -14\right)\right.\notag\\
   &&-24 \beta ^3-20 \beta ^2-18 \beta -16 n^4+n^3 (-48 \alpha -48 \beta-32)\notag\\
   &&+n^2 \left(-48 \alpha ^2+\alpha  (-96 \beta -72)-48 \beta ^2-88 \beta
   -36\right)\notag\\
   &&+n (-16 \alpha ^3+\alpha ^2 (-48 \beta -48)+\alpha  \left(-48 \beta
   ^2-128 \beta -60\right) \notag\\
   && -16 \beta ^3-80 \beta ^2-56 \beta-32)-17) \notag\\
   &&+2x
   \left(4 \alpha ^2+\alpha  (10 \beta +5)+6 \beta ^2+7 \beta +8 n^4+n^3 (24 \alpha +24 \beta
   +24)\right.\notag\\
   &&\left.+n^2 \left(24 \alpha ^2+\alpha  (48 \beta +48)+24 \beta ^2+48 \beta +28\right)\right.\notag\\
   &&+n
   \left(8 \alpha ^3+\alpha ^2 (24 \beta +24)+\alpha  \left(24 \beta ^2+48 \beta
   +32\right)\right.\notag\\
   &&\left.\left.+8 \beta ^3+24 \beta ^2+34 \beta +21\right)+5\right) 
   \notag\\
   &&+(1+2\beta)\left(4 \alpha ^2+\alpha  (8 \beta +6)+4 \beta ^2+6 \beta +4 n^2+n (8 \alpha +8 \beta +6)+3\right).\notag \\
\widehat{\mathcal{A}}_4(x;2n+1)& = & x^3 ( x^2-1)^3,\notag \\
\widehat{\mathcal{B}}_4(x;2n+1)& = &-2 x^2 (x^2-1)^2 \left((n-8) x^2-n+2\right) ,\notag \\
\widehat{\mathcal{C}}_4(x;2n+1)& = &- x (x^2-1)\Big(3 x^4(10n-19) 
+2 x^2 \left(2 \beta +2 n^2+(2 \beta -14) n+17\right)\notag\\
&&-4 \beta -4 n^2-(4 \beta +2) n-5
\Big) ,\notag \\
\widehat{\mathcal{D}}_4(x;2n+1)& = &
   x^6 \left(8 \alpha ^3+24 \alpha ^2 \beta +\alpha  \left(24 \beta ^2-8\right)+8 \beta ^3-8 \beta
   +8 n^3+n^2 (24 \alpha +24 \beta )\right.\notag\\
 &&  \left.+n \left(24 \alpha ^2+48 \alpha  \beta +24 \beta^2-122\right)-9\right)\notag\\
   &&+x^4 \Big(-24 \alpha ^2 \beta +\alpha  \left(-48 \beta ^2-16 \beta +8\right)-24 \beta ^3-16 \beta
   ^2-12 \beta \notag\\
   &&-16 n^3+n^2 (-40 \alpha -56 \beta -28)\notag\\
   &&+n \big(-24 \alpha ^2+\alpha 
   (-88 \beta -8)-64 \beta ^2-44 \beta+142\big)+19\Big)\notag\\
   &&+x^2 \Big(\alpha  \left(24 \beta ^2+16 \beta \right)+24 \beta ^3+32 \beta ^2+32 \beta +8 n^3\notag\\
   &&+n^2
   (16 \alpha +40 \beta +32) +n \left(\alpha  (40 \beta +8)+56 \beta ^2+64 \beta
   -14\right)+1\Big) \notag\\
   &&-(2 \beta +1) \left(4 \beta ^2+6 \beta +4 n^2+(8 \beta +6) n+3\right),\notag \\
\widehat{\mathcal{E}}_4(x;2n+1)& = & -16x (n+1) (\alpha +\beta +n+2)\Big(\notag\\
&&
+x^4 \left(\alpha ^2+\alpha  (2 \beta -2)+\beta ^2-2 \beta +n^2+n (2 \alpha +2 \beta -2)+3\right)\notag\\
&&+x^2 \left(\alpha  (1-2 \beta )-2 \beta ^2+\beta -2 n^2+n (-2 \alpha -4 \beta +1)-3\right)\notag\\
&&
+\beta ^2+\beta +n^2+(2 \beta +1) n+1
\Big).\notag 
\end{eqnarray}  

\section*{Conclusions}

The method developed in this work provides a unified and constructive approach for deriving structure relations and a fourth-order differential equation for any Laguerre–Hahn orthogonal family. Its symbolic implementation demonstrates its efficiency and versatility, as illustrated by the examples treated. The algebraic formalism presented here not only recovers known results for classical and semiclassical families but also offers an algorithm tool for handling more general Laguerre–Hahn sequences. In addition, new structure relations and differential equations have been obtained for both semiclassical and classical families.\\
Thus, the present work fills the existing gaps in the literature regarding explicit structure relations and fourth-order differential equations for Laguerre–Hahn families, while providing a reliable algorithm applicable to all such sequences.



\section*{Addendum}

In this paper, we present the results of a scientific work originally
started by the two last authors, Pascal Maroni and Zélia da Rocha.
This project coupled theoretical results
with their implementation in a {\it Mathematica$^{\circledR}$} software.
As is often made by scientists, the work was suspended with the 
intention to resume it later.
Pascal Maroni was a very prolific researcher, and other projects and collaborations
prevented him from returning to this subject.
Unfortunately, Pascal Maroni passed away in January 2024.
This was not the only work that Pascal began and failed to complete.
His friends, collaborators, and admirers thus decided to honor his memory,
to finish all the ongoing works and publish them to disseminate his ideas into
the scientific community, and to insert Pascal as one of the authors.
In this spirit, Zélia da Rocha invited the first author, Mohamed Khalfallah, 
to collaborate in the theoretical part so that this work could finally be completed.
 
\section*{Declarations}

\noindent  {\bf Data Availability} No datasets were generated or analysed during the current study.

\noindent  {\bf Conflicts of Interest} The authors have no conflict of interest to declare.

\noindent  {\bf Competing interests} The authors declare no competing interests.

\noindent {\bf Funding:} The third author was partially supported by CMUP, a member of LASI, which is financed by national funds through FCT -- Funda\c c\~ao para a Ci\^encia e a Tecnologia, I.P., under the projects with reference 
UID/00144/2025.

\bibliographystyle{plain}
\bibliography{sn-bibliography-4oDELH}

\end{document}